\documentclass{article}

\usepackage{amsmath,amsfonts,amssymb,amsthm,hyperref}
\usepackage{color}
\usepackage{tikz}
\usepackage{xcolor}

\newtheorem{thm}{Theorem}[section]
\newtheorem{cor}[thm]{Corollary}
\newtheorem{lem}[thm]{Lemma}
\newtheorem{prop}[thm]{Proposition}
\newtheorem*{prop*}{Proposition}

\theoremstyle{definition}

\newtheorem{conj}[thm]{Conjecture}
\newtheorem{question}[thm]{Question}
\newtheorem{defn}[thm]{Definition}

\newtheorem{example}[thm]{Example}

\newtheorem*{remark}{Remark}
\newtheorem*{note}{Note}

\newcommand{\Z}{\mathbb{Z}}
\DeclareMathOperator{\len}{len}

\begin{document}

\title{Free structures and limiting density} 

\author{Johanna N.\ Y.\ Franklin\thanks{The first author was supported in part by Simons Foundation Collaboration Grant \#420806.}, Meng-Che ``Turbo'' Ho\thanks{The second author acknowledges support from the National Science Foundation under Grant No.\ DMS-2054558.}
, and Julia Knight\thanks{The first and third authors acknowledge support from the National Science Foundation under Grant \#DMS-1800692.} \thanks{This material is based upon work supported by the National Science Foundation under Grant No.\ DMS-1928930 while the second and third authors participated in a program hosted by the Mathematical Sciences Research Institute in Berkeley, California, during Fall 2020 and Summer 2022 programs.}}

\maketitle 
\begin{abstract}
Gromov asked what a typical (finitely presented) group looks like, and he suggested a way to make the question precise in terms of \emph{limiting density}.  The typical finitely generated group is known to share some important properties with the non-abelian free groups.  We ask Gromov's question more generally, for structures in an arbitrary \emph{algebraic variety} (in the sense of universal algebra), with presentations of a specific form.  We focus on elementary properties.  We give examples illustrating different behaviors of the limiting density.  Based on the examples, we identify sufficient conditions for the elementary first-order theory of the free structure to match that of the typical structure; i.e., a sentence is true in the free structure iff it has limiting density $1$.    
\end{abstract}

\section{Introduction} 

In the paper where he introduced the notion of hyperbolic group, Gromov \cite{G} asked what a typical group looks like.  He was thinking of finitely presented groups.  He described, in terms of limiting density, what it might mean for a typical group to have some property $Q$, and he stated that the typical group is hyperbolic.  Gromov's notion has been made precise in different ways; see, for instance, the survey \cite{Olivier}.  Ol'shanskii \cite{Ol'shanskii} cleaned up the statement and the proof that the typical group is hyperbolic.  The third author conjectured that for the typical group obtained from a presentation consisting of $m\geq 2$ generators and a single relator, the elementary first order theory matches that of the non-abelian free groups.  In this paper, we generalize Gromov's question to arbitrary \emph{equational classes}, or \emph{algebraic varieties} (in the sense of universal algebra).  Here, as for groups, the notions of \emph{finite presentation} and \emph{free structure} make sense.  We find examples exhibiting different behavior.  Our main results are for varieties with only unary functions.  With some rather strong conditions on the variety, and on the presentations, we obtain the analogue of the conjecture for groups:  the sentences true in the free structure are exactly those with limiting density~$1$.  

\subsection{Organization} 

We begin with Gromov's original question, which concerned finitely presented groups. We then generalize this question to finitely presented members of a variety $V$.  Section \ref{sec:generalizing_the_question} has some background on algebraic varieties. Section~\ref{illu-ex} has four examples illustrating different possible behaviors and also introduces key ideas that will appear in later proofs. The first example shows that in a bijective structure with a single identity, the sentences true in the free structures have limiting density $1$. In the second example, the set of sentences with limiting density 1 is the theory of a structure in the variety, but this structure is not finitely generated, nor is it free. In the third example, we look at bijective structures as in the first example but with two identities, and we give a sentence for which the limiting density does not exist. In the fourth example, we look at abelian groups and give sentences where the limiting density is strictly between $0$ and $1$. 

In Section \ref{general}, we consider more general bijective varieties.  We show that for these varieties and presentations with a single generator and a single identity, a sentence has limiting density $1$ if and only if it is true in the free structure.  For a language with unary function symbols $f_1,\ldots,f_n$, injective and commuting, we say how to find elementary invariants.\footnote{We are grateful to Sergei Starchenko, who, having seen our result for bijective structures, suggested that we look further at varieties for languages with unary functions.}  We use a version of Gaifman's Locality Theorem, which we prove using saturation.\footnote{We are grateful to Phokion Kolaitis for alerting us to Gaifman's Theorem and providing references.}  In Section \ref{naming}, we consider sentences with constants.  For an arbitrary variety and presentations with a fixed tuple $\bar{a}$ of generators, we give conditions guaranteeing that, for sentences $\varphi$ in the language with added constants naming the generators, $\varphi$ has limiting density $1$ if and only if it is true in the free structure on $\bar{a}$.  In Section \ref{sec:MoreExamples}, we give further examples illustrating the results from Sections \ref{general} and \ref{naming}.   

\subsection{Gromov's question about groups}\label{subsec:GromovQ}

Here, we recall Gromov's original question and mention some prior work on typical or random groups.    
The usual language for groups has a binary operation symbol (for the group operation), a unary operation symbol (for the inverse), and a constant (for the identity).  Let $T$ be the theory of groups.  Recall that a \emph{group presentation} consists of a tuple $\bar{a}$ of \emph{generators} and a set $R$ of words $w_i(\bar{a})$ on these generators, called \emph{relators}.  In the group $G$ with presentation $\langle\bar{a}|R\rangle$, $G\models t(\bar{a}) = e$ if and only if $T\cup \{w(\bar{a}) = e:w(\bar{a})\in R\}\models t(\bar{a}) = e$.  Suppose $F$ is the free group on $\bar{a}$ and $N$ is the subgroup of $F$ consisting of the elements $t(\bar{a})$ such that $T\cup \{w(\bar{a}) = e\}\models t(\bar{a}) = e$.  Then $G\cong F/N$. 

The notion of limiting density depends not just on the variety, but also on the allowable group presentations, Ol'shanskii \cite{Ol'shanskii} considered presentations with $m$ generators and $k$ relators, all reduced.  Kapovich and Schupp \cite{KS} considered the case where $k = 1$.  In the Gromov ``density'' model, the number of relators may vary but is bounded in terms of the length of the relators and a parameter $d$.  It is important to bound the number of relators in some way; otherwise, the typical group will almost surely be trivial \cite{Ol'shanskii}.    
  
\begin{defn} 

Let $Q$ be a property of interest.  Let $P_s$ be the number of presentations in which the relators have length at most $s$, and let $P_s(Q)$ be the number of these presentations for which the resulting group has property $Q$.  The \emph{limiting density} for $Q$ is $\lim_{s\rightarrow\infty}\frac{P_s(Q)}{P_s}$, if this limit exists. 

\end{defn} 

We consider the \emph{typical group} to have property $Q$ if $Q$ has density $1$.  We are particularly interested in elementary properties.  For a sentence $\varphi$, the density of the property of satisfying $\varphi$ will be called simply the density of $\varphi$.  The typical group, in the sense of limiting density, is also called the \emph{random group}.\footnote{Harrison-Trainor, Khoussainov, and Turetsky \cite{HKT} took a different approach and considered \emph{random structures} more along the lines of the Rado graph.}  The typical group has some properties of free groups.  Gromov introduced the property of hyperbolicity and stated that the typical group is hyperbolic.  Ol'shanskii \cite{Ol'shanskii} showed that for presentations with $m$ generators and $k$ reduced relators, the property of being hyperbolic has limiting density $1$.  Kapovich and Schupp \cite{KS} showed that for presentations with $m$ generators and $1$ reduced relator, the property that all minimal generating tuples are Nielsen equivalent has limiting density $1$.  \emph{Nielsen equivalence} means that one tuple can be transformed into the other by a finite sequence of simple, obviously reversible, kinds of steps.     

Benjamin Fine, in conversation with the third author at the JMM in January of 2013, made an off-hand comment to the effect that in the limiting density sense, all groups look free.  Fine's comment gave rise to the conjecture below, saying that for presentations with $m\geq 2$ generators and $1$ relator, the typical group has the same elementary first order theory as the the free group.  By a result of Sela \cite{S} (see also Kharlampovich-Miasnikov \cite{KM}), the elementary first-order theory of all non-abelian free groups is the same.  The conjecture is given in \cite[Conjecture 2.2]{Ho}.    

\begin{conj}[Knight]
\label{Conjecture}

Take groups given by presentations with a fixed $m$-tuple of generators, for $m\geq 2$, and $1$ relator (the Kapovich-Schupp model).  For all elementary first order sentences $\varphi$, 
 
\begin{enumerate}

\item  the limiting density exists, and

\item  the density has value $1$ if $\varphi$ is true in the non-abelian free groups, and $0$ otherwise.  

\end{enumerate}

\end{conj}  

There is some evidence for the conjecture.  By a result of Arzhantseva and Ol'shanskii \cite{Ar96, Ar97}, a random group obtained from a presentation with $m$ generators and $k$ relators has many free subgroups.  Thus, an existential sentence true in the free group is also true in a random group.  Kharlampovich and Sklinos \cite{KhaSkl} used Gromov's density model, with parameter $d$.  In this setting, they showed the following.  
 
\begin{thm} [Kharlampovich and Sklinos \cite{KhaSkl}] 

A random group, in Gromov's density model with $d\le 1/16$, satisfies a universal sentence if and only if the sentence is true in the non-abelian free groups.

\end{thm}

The Kharlampovich-Sklinos result implies the conjecture (for universal sentences), but we will not give a proof here.

\section{Generalizing Gromov's question}\label{sec:generalizing_the_question}

The question that Gromov asked about groups makes sense for other algebraic varieties as well.  We begin by presenting our definition of an algebraic variety; then we discuss the types of presentations we will allow and give some basic lemmas.

\subsection{Algebraic varieties}

\begin{defn}  

A language is \emph{algebraic} if it consists only of function symbols and constants.

\end{defn}

The term ``algebraic variety'' is used to mean different things in algebraic geometry and in universal algebra.  The definition that we give  below is the one from universal algebra.  

\begin{defn}[Algebraic variety]

For a fixed algebraic language $L$, a class $V$ of $L$-structures is an \emph{algebraic variety}, or simply \emph{variety}, if it is closed under substructures, homomorphic images, and direct products.

\end{defn}

For our purposes, it is convenient to use the following equivalent definition, of ``equational class.'' 

\begin{defn} [Equational class]

For a fixed algebraic language $L$, a class $V$ of $L$-structures is an \emph{equational class} if it is axiomatized by sentences of the form $(\forall\bar{x}) t(\bar{x}) = t'(\bar{x})$---universal quantifiers in front of an equation.  

\end{defn}

Birkhoff showed that these two definitions are equivalent.  Mal'tsev defined a broader class of theories whose models have well-defined presentations.  See \cite{BS} for a general overview of universal algebra, where the result below appears as Theorem 11.9.      

\begin{thm}

For a fixed algebraic language $L$, a class of $L$-structures is an equational class if and only if it is a variety. 

\end{thm}

In the usual language for groups, namely $\{\cdot, ^{-1}, e\}$, the group axioms have the required form.  Thus, groups form a variety.

\bigskip

Now we consider an arbitrary algebraic variety $V$.  For a fixed generating tuple $\bar{a}$, there is a well-defined free structure $F$ generated by $\bar{a}$.  If $\mathcal{A}$ is a structure in $V$ generated by $\bar{a}$, then $\mathcal{A}$ is a quotient of $F$ under an appropriate equivalence relation on terms $t(\bar{a})$.  This equivalence relation becomes equality in $\mathcal{A}$.          

\begin{defn}

For a variety $V$, a \emph{presentation} has the form $\bar{a}|R$, where $\bar{a}$ is a generating tuple and $R$ is a set of identities on $\bar{a}$.  We write $\langle\bar{a}|R\rangle$ for the structure $\mathcal{A}$ such that the identities $t(\bar{a}) = t'(\bar{a})$ true in $\mathcal{A}$ are just the ones that follow logically from $R$ and the axioms for $V$.    

\end{defn}

We ask what the typical behavior is for members of a variety given by presentations of a specific form.  

\subsection{Allowable presentations}

In this paper, almost all of the languages of our varieties will be either the group language or a language with just unary function symbols.  We consider presentations with a fixed generating tuple $\bar{a}$, say of length $m$.  For the analogue of the Ol'shanskii setting, we consider presentations with $k$ identities for some fixed $k$.  For the analogue of the Kapovich-Schupp setting, we set $k = 1$.  This is the primary case we will consider.  Where we do consider $k > 1$, our presentations have the form $\bar{a}|R$, where $R$ is an \emph{unordered} set of identities.  

We may restrict the identities in certain natural ways.  For groups, we do what the group theorists do; that is, we suppose that the identities have the form $w(\bar{a}) = e$, where $w(\bar{a})$ is a word representing a product of various $a_i$ and $a_i^{-1}$.   
For the variety in the language consisting of two unary function symbols $S,S^{-1}$ with axioms saying that the two functions are inverses, we may restrict in a similar way, allowing only identities of the form $t(a_i) = a_j$; that is, with function symbols only on the left.  For the language with finitely many unary function symbols $f_1,\ldots,f_n$ and varieties that do not have axioms explicitly saying that one $f_j$ is the inverse of another $f_i$, our identities have the form $t(a_i) = t'(a_j)$, where $t(x)$ and $t'(x)$ are terms built up from the function symbols.   

\subsubsection{Length}

We will need to measure length of identities in our presentations. We will use the following conventions, based on the restrictions described above.

\begin{defn}\

\begin{itemize}

\item  In the setting of groups, the \emph{length} of an identity of the form $w(\bar{a}) = e$ is the number of occurrences of the various $a_i$ and $a_i^{-1}$ in the word $w(\bar{a})$.  This is the usual length of the relator.  

\item  For varieties in the language $L$ with just the unary function symbols $f_1,\ldots,f_n$, the \emph{length} of an identity of the form $t(a_i) = t' (a_j)$ is the total number of occurrences of the function symbols in the terms $t$ and $t'$.

\end{itemize} 

\end{defn}

\subsubsection{Limiting density}   

As for groups, we consider limiting density.   Here is the formal definition of limiting density.     

\begin{defn}

Fix a language, a variety, and a set of allowable presentations with an $m$-tuple $\overline{a}$ of generators and $k$ identities.  We write $P_s$ for the number of presentations in which all of the identities have length at most $s$.  For a property $Q$, let $P_s(Q)$ be the number of these presentations for which the resulting structure has property $Q$.  Then the \emph{limiting density} of $Q$ is $\lim_{s\rightarrow\infty}\frac{P_s(Q)}{P_s}$, provided that this limit exists.

\end{defn}

We are particularly interested in the case where $Q$ is the property of satisfying an elementary first-order sentence $\varphi$ in the language of the variety, possibly with added constants for the generators.  We write $P_s(\varphi)$ for the number of presentations in which the identities have length at most $s$, and the resulting structure satisfies $\varphi$.  We say that $\varphi$ has \emph{limiting density $d$} if $\lim\limits_{s\to \infty}\frac{P_s(\varphi)}{P_s} = d$. 

\begin{defn}

We say the variety $V$, with a specified set of allowable presentations, satisfies the \emph{zero--one law} if for every elementary first-order sentence $\varphi$ in $L$, $\varphi$ has limiting density $1$ or $0$.

\end{defn}

\subsubsection{Sets versus tuples of identities}

We have said that our presentations consist of a tuple of generators and an \emph{unordered} set of distinct identities.  Other possibilities would be to consider ordered tuples of identities, with or without repetition.  In practice, most of the time, we will consider a single identity.  When we do consider more than one identity, we show that the results would be the same for ordered tuples of identities allowing repetition, ordered tuples not allowing repetition, and unordered sets of identities.  

As above, we write $P_s$ for the number of unordered sets of $k$ identities of length at most $s$.  In the result below, we write $P^*_s$ for the number of ordered $k$-tuples allowing repetition, and $P^{**}_s$ for the number of ordered $k$-tuples not allowing repetition.       

\begin{prop}
\label{unorderedsets.v}

Let $N(s)$ be the number of identities in $L$ of length at most $s$ and suppose that $\lim\limits_{s\to \infty}N(s) = \infty$. Then for any sentence $\varphi$, 
\[\lim\limits_{s\to \infty}\frac{P_s(\varphi)}{P_s} = \lim\limits_{s\to \infty}\frac{P^*_s(\varphi)}{P^*_s} = \lim\limits_{s\to\infty}\frac{P^{**}_s(\varphi)}{P^{**}_s}\ .\] 

\end{prop}

\begin{proof} 

By definition, $P_s = {N(s) \choose k}$, $P^*_s = N(s)^k$, and $P^{**}_s = k!\cdot P_s$.     
Each unordered set of $k$ identities yields $k!$ ordered $k$-tuples of identities.  Thus, it is clear that 
\[\lim\limits_{s\to \infty}\frac{P_s(\varphi)}{P_s} = \lim\limits_{s\to \infty}\frac{P^{**}_s(\varphi)}{P^{**}_s}\ .\] 

To compare $\frac{P_s(\varphi)}{P_s}$ and $\frac{P^*_s(\varphi)}{P^*_s}$, we need the following:

\bigskip
\noindent
\textbf{Claim}:  

\begin{enumerate}

\item  $\frac{k!P_s}{P^*_s} \rightarrow 1$,

\item  $\frac{N(s)^k - k!P_s}{P^*_s}\rightarrow 0$. 

\end{enumerate}

\begin{proof} [Proof of Claim] 

For (1), the denominator is $N(s)^k$ and the numerator is a polynomial in $N(s)$ with leading term $N(s)^k$.  
For (2), the numerator is a polynomial in $N(s)$ of degree less than $k$, and the denominator is $N(s)^k$. 
\end{proof} 

Now, we note that  
\[k!P_s(\varphi) \le P^*_s(\varphi) \le k!P_s(\varphi)+ N(s)^k - k!{N(s) \choose k}.\]
Dividing by $P^*_s = N(s)^k$ and letting $s\to\infty$, we get
\[ \lim\limits_{s\to\infty} \frac{k!P_s(\varphi)}{P^*_s} \le \lim\limits_{s\to\infty} \frac{P^*_s(\varphi)}{P^*_s} \le \lim\limits_{s\to\infty} \frac{k!P_s(\varphi)+ N(s)^k - k!{N(s) \choose k}}{P^*_s}. \]
Using the claim, we see that the right-hand side is 
$$\lim\limits_{s\to\infty} \frac{k!P^*_s(\varphi)}{k!P^*_s} = \lim\limits_{s\to\infty} \frac{P^*_s(\varphi)}{P^*_s}.$$
\end{proof}

We can now phrase the questions we are interested in more formally.

\begin{question}\label{main-q}\

\begin{enumerate}

\item  Which varieties (with allowable presentations involving a fixed $m$-tuple $\bar{a}$ of generators) satisfy the zero--one law?  

\item  Given that the zero--one law holds, when do the sentences with limiting density $1$ match those true in the free structure? 

\end{enumerate} 

\end{question}

\subsection{Basic lemmas}  

Before we begin, we state three lemmas that hold very generally.

\begin{lem}

$(\varphi\vee\psi)$ has limiting density $0$ if and only if $\varphi$ and $\psi$ both have limiting density $0$; in fact, this holds for any finite disjunction.  

\end{lem}

\begin{proof}

We have $\frac{P_s(\varphi)}{P_s},\frac{P_s(\psi)}{P_s}\leq\frac{P_s(\varphi\vee\psi)}{P_s} \leq \frac{P_s(\varphi)}{P_s} + \frac{P_s(\psi)}{P_s}$.  From this, the lemma is clear.  
\end{proof}

\begin{lem}

$\varphi$ has limiting density $0$ just in case $\neg{\varphi}$ has limiting density $1$.  

\end{lem} 

\begin{proof}

We have $1 = \frac{P_s(\varphi)}{P_s} + \frac{P_s(\neg{\varphi})}{P_s}$.  Again, the lemma is clear.    
\end{proof} 

\begin{lem}

Let $S$ be the set of $L$-sentences with limiting density 1. Then $S$ is consistent and is closed under logical implication---if $\varphi_1,\ldots,\varphi_n\in S$ and $\varphi_1,\ldots,\varphi_n\vdash\psi$, then $\psi\in S$. 

\end{lem}

\begin{proof}

Suppose $S$ is not consistent.  By the Compactness Theorem, some finite subset is inconsistent. As every sentence in this set has density 1, there is a model of $T$ that realizes all these (finitely many) sentences, a contradiction.  By the previous two lemmas, we have that each $\neg\varphi_i$ has limiting density 0, so $\bigvee \neg\varphi_i$ also has limiting density 0, and so $\bigwedge\varphi_i$ has limiting density 1. But, $\bigwedge\varphi_i \vdash \psi$, so $\frac{P_s({\psi})}{P_s} \ge \frac{P_s(\bigwedge\varphi_i)}{P_s} = 1$, and the lemma follows. 
\end{proof}

\section{Illustrative examples}\label{illu-ex} 

In this section, we consider some of varieties and classes of presentations that illustrate different possibilities.  First, we consider the variety of bijective structures, and presentations with a single generator and a single identity in which the function symbols occur only on the left.  We show that the sentences true in the free structure are exactly those with limiting density $1$.  Second, we consider a variety of a single unary function and presentations with a single generator and a single identity.  Here, we show that a specific sentence that is true in the free structure has limiting density $0$.  Next, we again consider bijective structures, and presentations with a single generator but with two identities.  Here, we give sample sentences for which the limiting density does not exist.  Finally, we consider the variety of abelian groups, and presentations with a single generator and a single relator.  We give sentences for which the limiting density does not exist, and sentences for which the limiting density exists but is neither $0$ nor $1$.    

\subsection{Bijective structures}
\label{early-examples}  

We start with the variety of bijective structures.  Recall that the language consists of two unary function symbols $S,S^{-1}$.  The axioms are 
\[(\forall x)SS^{-1}(x) = x \mbox{\ \ and\ \ } (\forall x)S^{-1}S(x) = x\ .\]  
These guarantee that the function $S$ is $1-1$ and onto and that $S^{-1}$ is the inverse of $S$.  Let $T$ be the theory generated by these axioms.  The models consist of infinite $\mathbb{Z}$-chains and finite cycles $\mathbb{Z}_m$.  While these structures lack the mathematical interest and importance of groups, it is instructive to consider them because there are relatively simple elementary invariants, and for presentations with a single generator and a single identity, we can calculate the limiting densities for these sentences.  It turns out that the analogue of Conjecture \ref{Conjecture} holds.    

\begin{lem}

Over the theory $T$, every sentence is equivalent to a Boolean combination of sentences of the following basic types:

\begin{enumerate}

\item  $\alpha(n,k)$, saying that there are at least $k$ cycles of size $n$, 

\item  $\beta(n)$, saying that there is a chain of length at least $n$.

\end{enumerate}

\end{lem}

\begin{proof}

For any model $\mathcal{A}$ of $T$, we have a natural equivalence relation $\sim$ on the universe, where $a\sim b$ if $S^m(a) = b$ for some integer $m$.  Each $\sim$-class is a copy of $\Z$ or a finite cycle.  The isomorphism type of $\mathcal{A}$ is determined by the number of $\sim$-classes of different types.  Each model $\mathcal{A}$ of $T$ is elementarily equivalent to a saturated model $\mathcal{A}^*$, where $\mathcal{A}^*$ has infinitely many copies of $\Z$ if there is no finite bound on the sizes of the $\sim$-classes.  From this, we see that the isomorphism type of $\mathcal{A}^*$ and the elementary first order theory of $\mathcal{A}$ are determined by the sentences $\alpha(n,k)$ and $\beta(n)$.  
\end{proof}

We consider bijective structures with a single generator $a$.  There is a single $\sim$-class, which has the form $\Z$, an infinite chain, or $\Z_m$, a cycle of size $m$.  We note that in either $\Z$ or $\Z_m$, all elements are automorphic.  The following lemma is clear from the meanings of the sentences $\alpha(n,k)$ and $\beta(n)$.

\begin{lem}\

\begin{enumerate}

\item  For $k > 1$, $\alpha(n,k)$ is false in both $\Z$ and $\Z_m$, 

\item  $\alpha(n,1)$ is true only in $\Z_n$, 

\item  $\beta(n)$ is true in $\Z$; it is true in $\Z_m$ if and only if $m > n$. 

\end{enumerate}

\end{lem}

For models with a single generator, $\alpha(n,k)$ is false for $k > 1$, and $\beta(n)$ is equivalent to 
$\bigwedge_{m \leq n}\neg{\alpha(n,1)}$.  Thus, it is enough to consider the elementary invariants of the form $\alpha(n,1)$.   

\bigskip 

Here our presentations have a single identity, of the form $t(a) = a$ (function symbols occur only on the left).  We may refer to the term $t(a)$ as a \emph{relator}.  For a single relator $t(a)$, we get $\Z$ if for some $k$, $t(a)$ has $k$ occurrences of $S$ and $k$ occurrences of $S^{-1}$.  We get $\Z_m$ if for some $k$, $t(a)$ has either $m+k$ occurrences of $S$ and $k$ of $S^{-1}$ or $m+k$ occurrences of $S^{-1}$ and $k$ of $S$.

\bigskip    

We will show that for all $n\geq 1$, $\alpha(n,1)$ has limiting density~$0$.  This implies that $\neg \alpha(n,1)$, which is true in the free structure, has limiting density $1$.  We will use two combinatorial lemmas. The first is an approximation for $\left(\begin{array}{c}2k\\k\end{array}\right)$, good for large $k$.  The proof requires the use of Stirling's formula on all three factorials (see the website of Das \cite{D}).  

\begin{lem} 
\label{Fact 2} 

$\left(\begin{array}{c}2k\\k\end{array}\right) = (1+o(1)) \frac{2^{2k}}{\sqrt{\pi k}}$. 

\end{lem}

The second combinatorial lemma is an inequality.  

\begin{lem}\label{lemma:section2ineq}  

For all $n\geq 1$ and all $k$, $2\left(\begin{array}{c}n+2k\\k\end{array}\right) < \left(\begin{array}{c}n+ 2(k+1)\\k+1\end{array}\right)$.

\end{lem} 

\begin{proof}

Recall Pascal's Identity

$$\left(\begin{array}{c}n\\k\end{array}\right) + \left(\begin{array}{c}n\\k+1\end{array}\right) = \left(\begin{array}{c}n+1\\k+1\end{array}\right).$$
We prove the inequality by induction on $k$.  First, for $k = 0$, the inequality just says that $2<n+2$.  Now, suppose $k > 0$.  Applying Pascal's Identity to the right side of the inequality, we get 
\[\left(\begin{array}{c}
n+2(k+1)\\k + 1\end{array}\right) = \left(\begin{array}{c}
n+2k+1\\k\end{array}\right) + 
\left(\begin{array}{c}
n+2k+1\\k+1\end{array}\right)  \]
and then 
\[\left(\begin{array}{c}
n+2k\\k - 1\end{array}\right) +  
\left(\begin{array}{c}
n+2k\\k\end{array}\right) + 
\left(\begin{array}{c}
n+2k\\k\end{array}\right) + 
\left(\begin{array}{c}
n+2k\\k+1\end{array}\right)\ 
.\]
This is clearly greater than $2\left(\begin{array}{c}
n+2k\\k\end{array}\right)$.     
\end{proof}

To show that $\alpha(n,1)$ has limiting density $0$, we use several further lemmas.          

\begin{lem}  

$P_s = 2^{s+1} - 1$.

\end{lem}

\begin{proof}

The number of terms of length $m$ is $2^m$, so the number of terms of length at most $s$ is
$1 + 2 + \ldots + 2^s = 2^{s+1} - 1$.  
\end{proof}

\begin{lem}

$P_s(\alpha(m,1)) = \sum_{m+2k\leq s}2\left(\begin{array}{c}
m+2k\\k\end{array}\right)$. 

\end{lem}

\begin{proof}

For each $m\geq 1$, and each $k$, we have $\left(\begin{array}{c}
m+2k\\k\end{array}\right)$ terms with $m+k$ occurrences of $S$ and $k$ of $S^{-1}$, and the same number with the symbols switched.  
\end{proof}

The next lemma bounds the sum $P_{n+2k}(\alpha(n,1))$ by a single term.   

\begin{lem}\label{lem:BoundingOneTerm}

For all $n\geq 1$ and all $k\geq 0$, 
$P_{n+2k}(\alpha(n,1)) < \left(\begin{array}{c}n+2(k+1)\\k+1\end{array}\right) $.

\end{lem}

\begin{proof}

We fix $n$ and proceed by induction on $k$.  For $k = 0$, the left side is $P_n(\alpha(n,1)) = \left(\begin{array}{c}n\\0\end{array}\right) = 1$, and the right side is $\left(\begin{array}{c}n+2\\1\end{array}\right) = n+2 > 1$. Supposing that the statement holds for $k$, we prove it for $k+1$.  By Lemma \ref{lemma:section2ineq}, 
\[\left(\begin{array}{c}n+2(k+2)\\k+2\end{array}\right) > 
2\left(\begin{array}{c}
n+2(k+1)\\k+1\end{array}\right) = \left(\begin{array}{c}n+2(k+1)\\k+1\end{array}\right) + 
\left(\begin{array}{c}n+2(k+1)\\k+1\end{array}\right)\ .\]  
By the Induction Hypothesis, this is greater than 
\[\left(\begin{array}{c}n+2(k+1)\\k+1\end{array}\right) + P_{n+2k}(\varphi_n) = P_{n+2(k+1)}(\varphi_n)\ .\]       
\end{proof}

Now we can show that the limiting density of $\alpha(n,1)$ is $0$ for any $n\geq 1$. To do so, we must make an odd/even case distinction because the only way to get presentations of different lengths of the same structure is for these lengths to differ by a multiple of two, so $P_{n+2k+1}(\alpha(n,1))$ will equal $P_{n+2k}(\alpha(n,1))$. However, if $s = n+2k$ for some $k$, then $P_s(\alpha(n,1))$ has a new last term $\left(\begin{array}{c}n+2k\\k\end{array}\right)$, and $P_{s+1}(\alpha(n,1)) = P_s(\alpha(n,1))$.  Therefore, it is enough to show that 
\[\frac{P_{n+2k}(\alpha(n,1))}{P_{n+2k}}\rightarrow 0\ .\]  
By Lemma \ref{lemma:section2ineq}, the last term of $P_{n+2k}(\alpha(n,1))$ is greater than the sum of the earlier terms.  Thus, 
$P_{n+2k}(\alpha(n,1)) < 2\left(\begin{array}{c}
n+2k\\k\end{array}\right)$.  
Recall that $P_{n+2k} = 2^{n+2k+1} - 1$, which is strictly greater than $2^{n+2k}$, so $\frac{P_{n+2k}(\alpha(n,1))}{P_{n+2k}} < \frac{2}{2^{n+2k}}
\left(\begin{array}{c}
n+2k\\k\end{array}\right)$.  To prove that the limiting density of $\alpha(n,1)$ is $0$, it is enough to prove that

$$\frac{2}{2^{n+2k}}\left(\begin{array}{c}
n+2k\\k
\end{array}\right)\rightarrow 0.$$ 

We can express 
$\frac{2}{2^{n+2k}}
\left(\begin{array}{c}n+2k\\k\end{array}\right)$
as a product of two factors, one involving the fixed $n$ and the other not.
The first factor is 
\[\frac{2}{2^n}\left(\frac{2k+n}{k+n}\right)\left(\frac{2k+(n-1)}{k+(n-1)}\right) \ldots \left(\frac{2k+1}{k+1}\right)\ .\]  
This is an $(n+1)$-fold product with limit $2$ as $k\rightarrow\infty$.  The second factor is $\frac{1}{2^{2k}}\left(\begin{array}{c}2k\\k\end{array}\right)$.  By Lemma \ref{Fact 2} 
above, this is $(1+o(1))\frac{1}{\sqrt{\pi k}}$, which has limit $0$.   
All together, we have shown the following.

\begin{prop}
\label{alpha}

For $n\geq 1$, $\lim_{s\rightarrow\infty} \frac{P_s(\alpha(n,1))}{P_s} = 0$.  

\end{prop}   

From this, we get the following.    

\begin{thm}

For bijective structures with a single generator and a single relator, each sentence $\varphi$ has limiting density equal to $1$ if  $\ \Z\models\varphi$ and $\ 0$ otherwise.

\end{thm}

For later use, we state below another immediate consequence of Proposition~\ref{alpha}.  For a term $t(a)$, let $X$ be the difference between the number of occurrences of $S$ and the number of occurrences of $S^{-1}$ in $t(a)$.  The relators that make $\alpha(n,1)$ true are exactly those for which $|X| = n$.    

\begin{lem}
\label{later}

For each $k\in\mathbb{Z}$, the set of presentations $a,t(a) = a$ for which $X = k$ has density $0$.  

\end{lem}

\subsection{Structures with a single unary function} \label{basic-single-unary} 

Next, take the language with a single unary function symbol $f$, and the variety with no axioms.  We consider presentations with a single generator $a$ and a single identity.    
At first, we focus on the sentence $\varphi$ saying that $f$ is not $1-1$.  We will see that this sentence is false in the free structure, but it has limiting density $1$. After that, we will consider arbitrary sentences, and prove a zero-one law.

For an identity $f^{(r)}(a) = f^{(r')}(a)$, where $r,r'\geq 0$, the \emph{length} is $r + r'$.  If $r = r'$, then we get an $\omega$-chain.  If $0 < r < r'$, then we get a chain of length $r$ leading to a cycle of length $r' - r$.  If $0 = r < r'$, then we get a cycle of length $r'$.  Similarly, if $0 < r' < r$, we get a chain of length $r'$ leading to a cycle of length $r - r'$, and if $0 = r' < r$, we get a cycle of length $r$. 

We write $m + \Z_n$ for a chain of length $m$ leading to a cycle of length $n$, allowing the possibility that $m = 0$.  Any structure in our variety obtained from a single generator and a single identity has one of the following forms:

\begin{enumerate}
  
\item  an $\omega$-chain---this is the free structure, 

\item  a finite chain leading to a finite cycle, or

\item  a finite cycle.

\end{enumerate}           
  
\begin{lem}

$P_s = \frac{1}{2}(s^2 + 3s + 2)$. 

\end{lem}

\begin{proof}

The number of identities of length $m$ is $m+1$.   Then \\
$P_s = 1 + 2 + \ldots + (s+1) = \frac{1}{2}(s+2)(s+1) = \frac{1}{2}(s^2 + 3s + 2)$.   
\end{proof}     

Recall that $\varphi$ holds if the function is not $1-1$.  We can see that $\varphi$ is false in $\omega$ and the finite cycles $\Z_n$, and true in the structures $m+\Z_n$ for $m\geq 1$, so an identity $f^n(a) = f^{n'}(a)$ yields a structure in which $\varphi$ is false if $n = n'$ or if one of $n,n'$ is $0$, and true otherwise.  We can easily count the identities of length $m$ that make $\varphi$ false.  For $m = 0$, there is just one identity, and it makes $\varphi$ false.  For $m = 1$, there are two identities, and both make $\varphi$ false.  For $m\geq 2$, if $m$ is odd, there are just two identities that make $\varphi$ false, and if $m$ is even, there are three identities that make $\varphi$ false.  Thus, $P_s(\neg{\varphi}) = 1 + 2 + 3 + 2 + \ldots + (\frac52 + \frac12(-1)^s)$.      

\begin{prop}

The limiting density for $\varphi$ is $1$.  

\end{prop}

\begin{proof}

For $s\geq 1$, $P_s(\neg{\varphi})\leq 3s$, and $\lim_{s\rightarrow\infty}\frac{3s}{P_s} = 0$.  Therefore, $\neg{\varphi}$ has limiting density $0$, and $\varphi$ has limiting density~$1$.     
\end{proof}

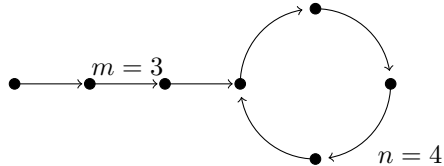
\begin{figure}[b]
\begin{center}
\begin{tikzpicture}
\coordinate (zero) at (-3,0);
\coordinate[label={$m=3$}] (middle) at (-1.5,0);
\coordinate(end) at (0,0);
\coordinate (loop) at (2,0);
\coordinate[label=315:{$n=4$}] (dots) at (1.707,-0.707);
\draw[fill] (zero) circle (2pt);
\draw[->] (zero) -- (-2.1,0);
\draw[->] (-1,0) -- (-.1,0);
\draw[->] (-2,0) -- (-1.1,0);
\draw[fill] (end) circle(2pt);
\draw[fill] (-2,0) circle (2pt);
\draw[fill] (-1,0) circle (2pt);
\draw[->] (0,0) arc (180:100:1cm);
\draw[->] (1,-1) arc (270:190:1cm);
\draw[->] (1,1) arc (90:10:1cm);
\draw[->] (2,0) arc (360:280:1cm);
\draw[fill] (loop) circle (2pt);
\draw[fill] (1,1) circle (2pt);
\draw[fill] (1,-1) circle (2pt);
\end{tikzpicture}
\end{center}
\caption{$m+\Z_n$ where $m = 3$ and $n = 4$}
\label{fig-fin-str}
\end{figure}

In fact, we will see that each sentence has limiting density $1$ or $0$.  The set of sentences with limiting density $1$ is not that of the free structure.  It is the theory of a structure $\mathcal{A}$ that we may think of as a limit of the structures $m+\Z_n$ (see Figure \ref{fig-fin-str}).  The limit structure $\mathcal{A}$ consists of an $\omega$-chain together with two $\omega^*$-chains and a single $\omega$-chain that come together at a special point---this point is the end of the two $\omega^*$-chains and the beginning of the $\omega$-chain (see Figure \ref{fig-lim-str}).  The chain of length $m$ is replaced, in the limit, by an $\omega$-chain plus one of the $\omega^*$-chains, and the $n$-cycle is replaced, in the limit, by the other $\omega$-chain and $\omega^*$-chain.  We note that $\mathcal{A}$ is not finitely generated.   

\begin{figure}[t]
\begin{center}
\begin{tikzpicture}
\coordinate (zero) at (-6,0);
\coordinate(end) at (0,0);
\coordinate[label={$\omega$}] (1) at (-4.5,0);
\coordinate[label={$\omega^*$}] (1) at (-1.25,0);
\coordinate[label=0:{$\omega^*$}] (1) at (0.25,-1.25,0);
\coordinate[label=0:{$\omega$}] (1) at (0.25,1.25,0);
\draw[fill] (zero) circle (2pt);
\draw[->] (zero) -- (-5.1,0);
\draw[fill] (-5,0) circle (2pt);
\draw[->] (-5,0) -- (-4.1,0);
\draw[fill] (-4,0) circle (2pt);
\draw[fill] (-3.75,0) circle (1pt);
\draw[fill] (-3.5,0) circle (1pt);
\draw[fill] (-3.25,0) circle (.5pt);
\draw[fill] (-3,0) circle (.5pt);

\draw[fill] (end) circle (2pt);
\draw[->] (-1,0) -- (-.1,0);
\draw[fill] (-1,0) circle(2pt);
\draw[fill] (-1.25,0) circle (1pt);
\draw[fill] (-1.5,0) circle (1pt);
\draw[fill] (-1.75,0) circle (.5pt);
\draw[fill] (-2,0) circle (.5pt);

\draw[->] (end) -- (0.18,0.9);
\draw[fill] (0.2,1) circle (2pt);
\draw[fill] (0.25,1.25) circle (1pt);
\draw[fill] (0.3,1.5) circle (1pt);
\draw[fill] (0.35,1.75) circle (.5pt);
\draw[fill] (0.4,2) circle (.5pt);

\draw[->] (0.2,-1) -- (0.02,-0.1);
\draw[fill] (0.2,-1) circle (2pt);
\draw[fill] (0.25,-1.25) circle (1pt);
\draw[fill] (0.3,-1.5) circle (1pt);
\draw[fill] (0.35,-1.75) circle (.5pt);
\draw[fill] (0.4,-2) circle (.5pt);
\end{tikzpicture}
\end{center}
\caption{The limit structure}
\label{fig-lim-str}
\end{figure}
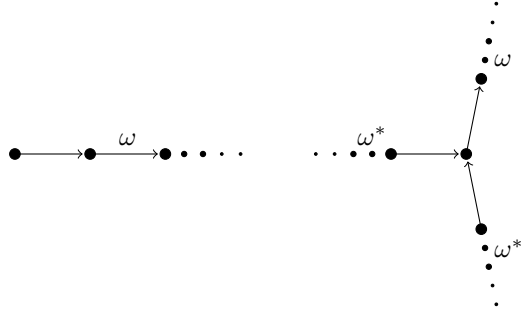

\begin{lem}

The theory of $\mathcal{A}$ is generated by the following sentences; note that the elements $a$ and $c$ are defined and not named by constants. 

\begin{enumerate}

\item  $\psi_a$, saying that there is a unique element $a$ with no $f$-pre-image, 

\item  $\psi_c$, saying that there is a unique element $c$ with two $f$-pre-images, 

\item  $\psi$, saying that there is no element with more than two $f$-pre-images,

\item  for every $n\in \omega$, $\alpha_n$ saying that there is no cycle of length $n$, and

\item  for every $n\in \omega$, $\beta_n$ saying that $a$ and $c$ are not connected by a chain of length $n$.

\end{enumerate}

\end{lem}

\begin{proof}

In any structure in our variety, there is an equivalence relation $\sim$, where $x\sim y$ if there is a finite sequence $x_0,\ldots,x_n$ such that $x = x_0$, $y = x_n$, and for each $i < n$, either $f(x_i) = x_{i+1}$ or $f(x_{i+1}) = x_i$.  If $\mathcal{B}$ and $\mathcal{C}$ are two saturated models of the sentences above, of the same cardinality, then $\mathcal{B}\cong\mathcal{C}$.  We map the special elements $a$ and $c$ in $\mathcal{B}$ to the corresponding elements of $\mathcal{C}$.  The $\sim$-class of $a$ is an $\omega$-chain that does not include $c$.  The $\sim$-class of $c$ has two $\omega^*$-chains ending in $c$ and one $\omega$-chain starting with $c$.  The other $\sim$-classes, not containing $a$ or $c$, are $\Z$-chains.  There are the same number of these in $\mathcal{B}$ and $\mathcal{C}$.
\end{proof}

\begin{lem}

The sentences above that generate the theory of $\mathcal{A}$ all have limiting density $1$. 

\end{lem} 

\begin{proof}

The sentence $\psi$, saying that there is no element with more than two $f$-pre-images, is true in all of the models that we get from a single generator $a$ and a single identity, so this has limiting density $1$.  We have seen that the set of presentations that give models in which $f$ is not $1-1$ has limiting density~$1$.  The models have the form $m+\Z_n$ for $m,n > 0$.  The sentences $\psi_a$, saying that there is a unique element $a$ with no $f$-pre-image, $\psi_c$, saying that there is a unique element $c$ with two $f$-pre-images, and $\psi$ are true in all of these models, so these sentences have limiting density $1$.     

Consider $\alpha_n$, saying that there is no cycle of length $n$.  Let $B$ be the set of presentations that make $\alpha_n$ false---the resulting model \emph{has} a cycle of length $n$.  These identities have the form $f^{n+k}(a) = f^k(a)$ or $f^k(a) = f^{n+k}(a)$.  The number of such identities of length $m$ is $0$ if $m - n$ is odd and $2$ if $m - n$ is even.  Then $P_s(\alpha_n) \leq 2s$.  Since $P_s = O(s^2)$, the limiting density of $B$ is $0$, so the density of $\alpha_n$ is $1$.     
Finally, consider $\beta_n$, saying that $a$ and $c$ are not connected by a chain of length $n$.  Let $C$ be the set of presentations that make this false.  These identities have the form $f^n(a) = f^{n+k}(a)$ or $f^{n+k}(a) = f^n(a)$.  The number of such identities of length $m$ is $2$ for all $m\geq 2n$.  Then $P_s(C)\leq 2s$.  Again the limiting density of $C$ is $0$, so the density of $\beta_n = 1$.  
\end{proof}

We conclude the following.

\begin{prop}

For all sentences $\varphi$ true in $\mathcal{A}$, the limiting density is $1$.  

\end{prop} 

\subsection{Bijective structures with two identities} 
\label{bij-two-id}

In the next example, we return to the variety of bijective structures as in Section~\ref{early-examples}.  As before, our presentations have a single generator $a$, but there are two identities instead of just one.  We will show that the limiting density need not exist.  Recall that the language consists of unary function symbols $S, S^{-1}$ and that the axioms say that $S$ and $S^{-1}$ are inverses.  The identities have the form $t(a) = a$, with function symbols only on the left.  Each identity is thus equivalent to one of the form $S^m(a) = a$, where $m\in\mathbb{Z}$.  

\begin{prop}

For bijective structures, and presentations with a single generator $a$ and two identities, the sentence $\varphi$ saying that the structure is a $1$-cycle does not have a limiting density.     

\end{prop}

The proof is somewhat involved.  We begin with some elementary lemmas, but eventually we will consider a random walk on a group and appeal to results from random group theory that depend on the Central Limit Theorem.  The lemma below tells us when the sentence $\varphi$ is true.     

\begin{lem}

Let $\mathcal{A}$ be the structure given by an unordered set consisting of two identities, equivalent to $S^m(a) = a$ and $S^{m'}(a) = a$.  Then $\mathcal{A}$ is a $1$-cycle if and only if $GCD(m,m') = 1$.   

\end{lem}

\begin{proof}

Note that $S^k(a) = a$ if and only if $a = S^{-k}(a)$.  Thus, we may suppose that both $m,m'\geq 0$.  First, suppose that $GCD(m,m') = 1$.  In this case, there are $r,s \in \Z$ such that $mr+m's = 1$. Then we have 
\[S(a) = S^{mr+m's}(a) = S^{mr}(S^{m's}(a)) = S^m\circ \cdots\circ S^m\circ S^{m'}\circ \cdots\circ S^{m'}(a) = a,\] 
so $\mathcal{A}$ is a $1$-cycle.  Now, suppose that $\mathcal{A}$ is a $1$-cycle, and let $GCD(m,m') = d$.  The axioms of $T$ and the identities $S^m(a) = a$ and $S^{m'}(a) = a$ are both satisfied in a $d$-cycle, and $\mathcal{A}$ can only be a $1$-cycle if $d =1$.               
\end{proof}                                                                                                                                                                                                                                                                                                                                                                                                                                         

Our presentations have two identities, but we also need some facts about single identities.  We indicate with $'$ that we are considering single identities, writing $P'_{=m}$ for the number of identities of length $m$ and $P'_s$ for the number of length at most $s$, and writing $P'_{=m}(B)$, $P'_s(B)$ for the number of these identities in a set $B$.  We reserve $P_s$ for the number of unordered pairs of identities of length at most $s$, and we write $P_s(B^2)$ for the number such that both identities are in $B$.  

For a single identity of the form $t(a) = a$, let $X$ be the difference between the number of occurrences of $S$ and the number of occurrences of $S^{-1}$ in $t$. 
Intuition may suggest that the statement $n|X$ should have limiting density $\frac{1}{n}$. This turns out to be true for odd $n$.  However, for $n = 2$, we find that the limiting density for the statement $2|X$ does not exist.  Essentially, the reason is that the last term of $P'_s(2|X) = \sum_{m\leq s} P'_{=m}(2|X)$ may be greater than the sum of all earlier terms, and this term depends on the parity of $s$.  The lemma below says what happens to $P_{=s}(2|X)$ as the parity of $s$ changes.   

\begin{lem}\

\begin{itemize}
\label{parity}
 
\item For even $m$, all identities of length $m$ satisfy $2\mid X$; none satisfies $2\nmid X$.

\item  For odd $m$, all identities of length $m$ satisfy 
$2\nmid X$ and none satisfies $2\mid X$.  

\end{itemize} 

\end{lem}

\begin{proof}

For $m = 0$, there is just one identity, and for this identity, $X = 0$.  Supposing that the statements hold for $m$, if $t$ has length $m$, then $t$ has two extensions of length $m+1$, and the parity of $X$ changes.   
\end{proof} 

The next lemma gives the proportion of single identities of length at most $s$ for which $2|X$ holds.  The value depends on the parity of $s$.  
 
\begin{lem}\

\begin{enumerate}\label{even-and-odd}
 
\item  $\displaystyle\lim\limits_{s\rightarrow\infty} \frac{P'_{2s}(2\mid X)}{P'_{2s}} = \frac{2}{3}$ and  
$\displaystyle\lim\limits_{s\rightarrow\infty}\frac{P'_{2s}(2\nmid X)}{P'_{2s}} = \frac{1}{3}$.

\item   $\displaystyle\lim_{s\rightarrow\infty}\frac{P'_{2s+1}(2\mid X)}{P'_{2s+1}} = \frac{1}{3}$ and 
$\displaystyle\lim_{s\rightarrow\infty}\frac{P'_{2s+1}(2\nmid X)}{P'_{2s+1}} = \frac{2}{3}$.  

\end{enumerate}

\end{lem}

\begin{proof}

The calculation is based on Lemma \ref{parity}.  For (1), the even case, we have $P'_{2s} = 2^{2s+1} - 1 = 4^s\cdot 2 - 1$, and
$$P'_{2s}(2\mid X) = \sum_{m\leq s} 2^{2m} = \sum_{m\leq s} 4^m = \frac{4^{s+1} - 1}{3}.$$ 
Then
$\displaystyle\frac{P'_{2s}(2\mid X)}{P'_{2s}} = \frac{4^{s+1} - 1}{3(4^s\cdot 2 - 1)}
\rightarrow\frac{2}{3}$. 

\bigskip

For (2), the odd case, we have $P'_{2s+1} = 2^{2s+2} - 1 = 4^{s+1} - 1$ and
$$P'_{2s+1}(2\mid X) = P'_{2s}(2\mid X) = \frac{4^{s+1} - 1}{3}.$$
Then 
$\displaystyle\frac{P'_{2s+1}(2 \mid X)}{P'_{2s+1}} = \frac{4^{s+1} - 1}{3\cdot(4^{s+1} - 1)} = \frac{1}{3}$.
\end{proof}

So far, the lemmas have involved only elementary calculations.  The next result is from random group theory \cite{CDDHS, SC}, concerning a random walk on a group.  The elements of the group $G = \Z_n$ represent the possible remainders after division of an integer $z$ by $n$.  In general, for a random walk, there are finitely many states, and given just the current state $s$, with no more prior history, we have fixed probabilities of passing next to state $s'$. We allow $s' = s$.  

Our states are group elements.  We write $\mu(g)$ for the probability of going in one step from the identity to $g$, and we write $\mu^{(k)}(g)$ for the probability of going in $k$ steps from the identity to $g$.  For the result below, the probability measure $\mu$, defined on $G$, is \emph{supported} on a special generating set $\Sigma$.  For $\mu$ to be supported on $\Sigma$ means that $\mu$ assigns non-zero probability to the elements of $\Sigma$, i.e., the set $\Sigma$ will consist of the group elements reachable from the identity in one step.  For any $n$, $\Sigma$ is also the set of differences $g' - g$, where $g'$ is a successor of $g$ reachable in one step.   The values of $\mu^k(g)$, for $k > 0$ are obtained by considering the tree of $1$-step extensions of length $k$ starting from the identity.  We multiply probabilities along the paths, and then sum over the paths leading to $g$.  

The result below tells us that the probability of each remainder $g$ in $\Z_n$ is approximately $\frac{1}{n}$, and that the convergence (as $k\rightarrow\infty$) has a great deal of uniformity. 

\begin{thm}[{\cite[Theorem 7.3]{SC}}]  
\label{SC}

There exist $\alpha,\beta > 0$ such that for any group $G$ of the form $\Z_n$, any generating set $\Sigma$ containing the group identity element, and a probability measure $\mu$ supported on $\Sigma$, we have that for all $g\in G$ and all $k\in\omega$, 
\[\left| \mu^{(k)}(g) - \frac{1}{n} \right| < \alpha e^{-(\frac{\beta k}{n^2})}\]

\end{thm} 

To adapt this theorem to our setting, we will consider an odd $n$ and identities of even length $m$. We break the identity into pieces of length 2, so each piece has $1/4$ chance of being each of $SS$, $S^{-1}S^{-1}$, $SS^{-1}$, or $S^{-1}S$. These correspond to $2, -2$, and $0$ in the random walk on $\Z_n$, and when $n$ is odd, $\Sigma = \{-2, 0, 2\}$ generates $\Z_n$.  Just as the identities of even length approach a uniform distribution, so do the identities of odd length, and, consequently, so do the identities of length at most $s$.  We have the following:

\begin{cor}\label{cor:OddnAlls}

For any odd number $n$ and any $s$,  \\
$$\displaystyle\lim_{s\to\infty}\frac{P'_{s}(n\mid X)}{P'_{s}} = \frac{1}{n} \text{ and } \displaystyle\lim_{s\to\infty}\frac{P'_{s}(n\nmid X)}{P'_{s}} = \frac{n-1}{n}.$$  

\end{cor}

Our presentations have an unordered set of two identities.  However, it is easier to count ordered pairs, allowing repetition.  By Proposition \ref{unorderedsets.v}, we get the same limiting densities, so we will count ordered pairs allowing repetition of elements instead.  Let $C$ be the set of presentations in which the difference functions $X_1,X_2$ are both divisible by some prime $p$.  It follows from Lemma \ref{later} that $X_1 = 0$, and $X_2 = 0$ both have limiting density $0$.  The important part of $C$ consists of the presentations such that $X_1\not= 0$ and $X_2\not= 0$, and in what follows, we write $C$ for this important part.  For the presentations in $C$, there is some prime that divides both $X_1$ and $X_2$, and both $X_1,X_2$ are non-zero.  

\begin{defn}\label{ps}

For each $s$, let $p_s$ be the greatest prime $p$ such that \\
$2\cdot 3\cdots p\leq ln(s)$.

\end{defn} 

For every $s$, we split $C$ into two parts, $C_1 = C_{s,1}$ and $C_2 = C_{s,2}$. Note that this splitting depends on $s$.  For a presentation in $C$, let $d$ be the least prime that divides both $X_1,X_2$.  Then the presentation is in $C_{s,1}$ if $d \leq p_s$ and in $C_{s,2}$ if $d > p_s$.  We will show that $C_{s,2}$ has limiting density $\lim\limits_{s\to \infty} \frac{P_s(C_{s,2})}{P_s} = 0$ and that the limiting density of $C_1$ does not exist---it toggles between two values, one for even 
$s$ and the other for odd $s$.  Among the primes, $2$ behaves differently from the odd primes.  We have shown that, for single identities with difference $X$, the limiting density of $2|X$ does not exist. We will see later that this explains why the limiting density of $C_1$ does not exist.  

\bigskip

Our first goal is to show that $C_2$ has limiting density $0$.  Toward this, we consider a single identity with difference $X$.

\begin{lem}

For each odd prime $p$ and all $s$, $\displaystyle\frac{P'_s(p|X\ \&\ X\not= 0)}{P'_s}\leq\frac{2}{p+1}$.

\end{lem}

\begin{proof}
We will first prove that $\displaystyle\frac{P'_{=m}(p|X\ \&\ X\not= 0)}{P'_{=m}}\leq\frac{2}{p+1}$ for all $m$.
Note that the numbers $P'_{=m}(X = n)$, for $-m\leq n\leq m$, form a Pascal triangle.  At the top, for $m = 0$, we have $1$, corresponding to $X = 0$.  For $m = 1$, we have $1$'s corresponding to $X = \pm 1$.  In general, for even $m$, $X$ takes the even values $n$ in the interval $[-m,m]$, and for odd $m$, $X$ takes the odd values in the interval $[-m,m]$.  In both cases, $P'_{=m+1}(X = n) = P'_{=m}(X = n - 1) + P'_{=m}(X = n + 1)$.  We can see that $P'_{=m}(X = n)$ decreases as $|n|$ increases, and that $P'_{=m}(X = n) = P'_{=m}(X = -n)$.  

\bigskip

For odd $m$ (so that $X$ is odd), we have 
\[P'_{=m}(X = \pm 1) \ge P'_{=m}(X = \pm 3) \ge \cdots \ge P'_{=m}(X = \pm p)\ \,\]
\[P'_{=m}(X = \pm (p+2)) \ge P'_{=m}(X = \pm (p+4)) \ge \cdots \ge P'_{=m}(X = \pm3p)\]
\[\cdots\]
Note that there are $p$ terms in each of the lines, except the first line, which has only $\frac{p+1}{2}$ terms. Therefore, we have 
\begin{align*}
1 = \sum\limits_{n:\text{odd}} \frac{P'_{=m}(X = n)}{P'_{=m}} &= \frac{P'_{=m}(X = \pm1)}{P'_{=m}}+\frac{P'_{=m}(X = \pm3)}{P'_{=m}} +\cdots\\
&\ge \left(\frac{p+1}{2}\right)\frac{P'_{=m}(X = \pm p)}{P'_{=m}} + p\cdot\frac{P'_{=m}(X = \pm 3p)}{P'_{=m}} +  
\cdots\\
&\ge \left(\frac{p+1}{2}\right)\left(\frac{P'_{=m}(X = \pm p)}{P'_{=m}} + \frac{P'_{=m}(X = \pm 3p)}{P'_{=m}} + \cdots\right) \\
& = \left(\frac{p+1}{2}\right)\frac{P'_{=m}(p \mid X \ \& \ X\neq 0)}{P'_{=m}}.
\end{align*}

If $m$ is even (so that $X$ is even), then we have 
\[P'_{=m}(X = \pm 2) \ge P'_{=m}(X = \pm 4) \ge \cdots \ge P'_{=m}(X = \pm 2p)\]
\[P'_{=m}(X = \pm (2p+2)) \ge P'_{=m}(X = \pm (2p+4)) \ge \cdots \ge P'_{=m}(X = \pm4p)\]
\[\cdots\]
In this case, each line has $p$ terms, and we get the following slightly stronger inequality:
\begin{align*}
1 = \sum\limits_{n:\text{even}} \frac{P'_{=m}(X = n)}{P'_{=m}} &= \frac{P'_{=m}(X = 0)}{P'_{=m}}+
\frac{P'_{=m}(X = \pm2)}{P'_{=m}}+\frac{P'_{=m}(X = \pm4)}{P'_{=m}}+\cdots\\
&\ge \frac{P'_{=m}(X = \pm2)}{P'_{=m}}+\frac{P'_{=m}(X = \pm4)}{P'_{=m}}+\cdots\\
&\ge p\cdot \frac{P'_{=m}(X = \pm 2p)}{P'_{=m}} + p\cdot \frac{P'_{=m}(X = \pm 4p)}{P'_{=m}} + \cdots\\
& = p\cdot \frac{P'_{=m}(p \mid X \ \& \ X\neq 0)}{P'_{=m}}.
\end{align*}
Combining the even and odd case, we get the desired $\frac{P'_{=m}(p|X\ \&\ X\not= 0)}{P'_{=m}}\leq\frac{2}{p+1}$.

Now, we turn our attention back to the inequality in the lemma, which concerns identities up to a certain length.  The quotient $\frac{P'_s(p\mid X \ \& \ X \neq 0)}{P'_s}$ is a weighted average (weighted by the proportion of identities of each length) of the probabilities $\frac{P'_{=m}(p\mid X \ \& \ X \neq 0)}{P'_{=m}}$, where $m \le s$. Thus, the lemma follows from the inequality on identities of a fixed length $\frac{P'_{=m}(p|X\ \&\ X\not= 0)}{P'_{=m}}\leq\frac{2}{p+1}$.
\end{proof}

We are now ready to consider both identities.

\begin{lem}
$\displaystyle \lim\limits_{s\to\infty}\frac{P_s(C_2)}{P_s} =0$.
\end{lem}

\begin{proof}

Below, we will appeal to Proposition \ref{unorderedsets.v} and consider, for each $s$, the probability space consisting of the ordered pairs of identities, each of length at most $s$.  Then the random variables $X_1,X_2$ are independent.  Counting ordered pairs of identities and allowing repetition, we see that for each $s$, 
\begin{align*}
P_s(C_2) &\leq\sum\limits_{p > p_s}P_s(X_1,X_2\not= 0\ \&\ p|X_1\ \&\ p|X_2)\\ 
 &= \sum\limits_{p > p_s}P'_s(X_1 \neq 0\ \&\ p|X_1)P'_s(X_2 \neq 0\ \&\ p|X_2).
\end{align*}

So, it follows from the previous lemma that
\begin{align*}
\frac{P_s(C_2)}{P_s} & \leq \sum\limits_{p > p_s}\frac{P'_s(X_1 \neq 0\ \&\ p|X_1)}{P'_s}\frac{P'_s(X_2 \neq 0\ \&\ p|X_2)}{P'_s}\\
& \leq \sum_{p > p_s} \left(\frac{2}{p+1}\right)^2
\end{align*}

By a well-known fact from number theory, the sum of the squares of the reciprocals of primes (or of all natural numbers) converges. Since $\lim\limits_{s\to \infty}p_s = \infty$, we have that $\lim\limits_{s\to\infty}\sum\limits_{p > p_s} (\frac{2}{p+1})^2 = 0$.  
Thus, $C_2$ has limiting density $0$.  
\end{proof}

We turn to $C_1$.  Again, we consider first a single identity.   

\begin{lem}
\label{squareroot} 

We write $D_s$ for the set of identities of length at most $s$ but greater than $\sqrt{s}$.  Then $\frac{P'_s(D_s)}{P'_s}\rightarrow 1$.  

\end{lem}

\begin{proof}

We have $P'_s(D_s) = P'_s - P'_{\sqrt{s}}$, and $\frac{P'_s(D_s)}{P'_s} = 1 - \frac{2^{\sqrt{s}+1} - 1}{2^{s+1} - 1}\rightarrow 1$.       
\end{proof}

Lemma \ref{squareroot} may be interpreted as saying that most identities of length at most $s$ have length at least $\sqrt{s}$.  We write $P_s(D_s^2)$ for the number of pairs of identities of length at most $s$ such that both have length at least $\sqrt{s}$.  The next lemma says that for most pairs of identities of length at most $s$, the length of both is at least $\sqrt{s}$. 
 
\begin{lem}

$\lim_{s\rightarrow\infty}\frac{P_s(D_s^2)}{P_s}\rightarrow 1$.  

\end{lem}       

Now, $\frac{P_s(C_1)}{P_s}$ is the probability that, among pairs of identities of length at most $s$, with difference functions $X_1$ and $X_2$, there is some prime $p \leq p_s$ such that $p|X_1\ \&\ p|X_2$.  We may suppose that both identities have length greater than $\sqrt{s}$.  We have seen that the limiting probability that $2$ divides both $X_1,X_2$ does not exist---for even $s$, it approaches $\frac{4}{9}$, while for odd $s$, it approaches $\frac{1}{9}$ \ref{cor:OddnAlls}.  Here, we consider odd primes.  We have justified thinking of the random variables $X_1,X_2$ (for identities of length at most $s$) as independent.  

We would like to assume that for $i = 1,2$, the events $p|X_i$ for \emph{different} primes $p$ are independent.  This turns out to be ``approximately'' true.  The probability that $X_1,X_2$ are \emph{not} both divisible by $3$ is approximately $1 - \frac{1}{3^2}$.  The probability that $X_1,X_2$ are \emph{not} both divisible by $3$ and not both divisible by $5$ is approximately $(1 - \frac{1}{3^2})(1 - \frac{1}{5^2})$.  The probability that $X_1,X_2$ are \emph{not} both divisible by any odd prime $p\leq p_s$ is approximately $\prod\limits_{3 \le p \leq p_s}(1 - \frac{1}{p^2})$.  This formula matches what we would get by laborious inclusion-exclusion counting. 

In fact, the divisibilities of $X_i$ by different primes may \emph{not} be independent.  However, we can apply the Chinese Remainder Theorem and consider the residue of $X_i$ modulo $N_s = \prod\limits_{p \le p_s} p$, which is $2\prod\limits_{3\le p\leq p_s}p$.  It follows from the Definition of $p_s$ (Definition \ref{ps}) that $N_s \le \ln(s)$.  This is where the random walk on the group comes in.  We will use Theorem \ref{SC}.  

\begin{thm}\label{two-id-thm} 

Below, we let $p$ range over all primes:

\begin{enumerate}

\item  $\displaystyle \frac{P_{2s}(C_{s,1}^c)}{P_{2s}}\rightarrow 
 \frac{5}{9}\cdot\prod_{3\le p}\left(1 - \frac{1}{p^2}\right)$.  
 
 \item  $\displaystyle \frac{P_{2s+1}(C_{s,1}^c)}{P_{2s+1}}\rightarrow
 \frac{8}{9}\cdot\prod_{3 \le p}\left(1 - \frac{1}{p^2}\right)$.  
  
  \end{enumerate}

 \end{thm}

\begin{proof}

We prove (1).  Take $\alpha,\beta > 0$ as in Theorem \ref{SC}.  Recall that $X_1$ and $X_2$ are the difference functions associated with the first and second identities, where in each identity, the function symbols are all on the left.   Fixing $s$, we consider the residue of $X_1$ and $X_2$ modulo $N_s = 2\prod\limits_{3 \le p\leq p_s} p$, where $N_s$ is at most $\ln(s)$.  For a single identity, we consider a string $t$ of length $k$.  For the fixed $s$, let $G_s$ be the group of possible remainders after division of $X_1$ by $N_s$.  Theorem~\ref{SC} tells us that for every $0 \le a < N_s$ and every $k$, 
$$\left|\frac{P'_{=k}(X_1 = a\pmod{N_s})}{P'_{=k}} - \frac2{N_s}\right| < \alpha e^{\frac{-\beta k}{N_s^2}}\text{ if $a$ and $k$ have the same parity, and}$$ 
$$\frac{P'_{=k}(X_1 = a\pmod{N_s})}{P'_{=k}} = 0 \text{ if $a$ and $k$ have different parities}.$$  
When we sum up the identities of length at most $s$, the previous lemma says that most of them will have length some $k \ge \sqrt{s}$. Thus, we may assume that $k \ge \sqrt{s}$.  The previous inequality yields
$$\left|\frac{P'_{=k}(X_1 = a\pmod{N_s})}{P'_{=k}} - \frac2{N_s}\right| < \alpha e^{\frac{-\beta k}{N_s^2}} < \alpha e^{\frac{-\beta\sqrt{s}}{\ln(s)^2}}.$$  

By Lemma \ref{even-and-odd}, among the identities of length up to $2s$, $2/3$ of them are even, and $1/3$ of them are odd. The probability $\frac{P'_{2s}(X_1 = a\pmod{N_s})}{P'_{2s}}$ is a weighted sum of $\frac{P'_{=k}(X_1 = a\pmod{N_s})}{P'_{=k}}$. Noticing that the rest of the previous inequality does not depend on $k$, doing a weighted sum gives
$$\left|\frac{P'_{2s}(X_1 = a\pmod{N_s})}{P'_{2s}} - \frac23\cdot\frac2{N_s}\right| < \alpha e^{\frac{-\beta\sqrt{s}}{\ln(s)^2}} \text{ if }a\text{ is even, and}$$
$$\left|\frac{P'_{2s}(X_1 = a\pmod{N_s})}{P'_{2s}} - \frac13\cdot\frac2{N_s}\right| < \alpha e^{\frac{-\beta\sqrt{s}}{\ln(s)^2}} \text{ if }a\text{ is odd.}$$

By the independence of $X_1,X_2$, we see that for all sufficiently large $s$, 
$$\left|\frac{P_{2s}(X_1 = a \pmod{N_s}\ \&\ X_2 = b\pmod{N_s})}{P_{2s}}- c\frac{4}{N_s^2}\right| < \alpha e^{\frac{-\beta\sqrt{s}}{\ln(s)^2}}$$
where $c = 
\begin{cases}
\frac19 & \mbox{if $a,b$ are both odd,}\\
\frac29 & \mbox{if one of $a,b$ is odd and the other is even,}\\
\frac49 & \mbox{if $a,b$ are both even.}\\
\end{cases}$

\bigskip

Now, we consider pairs $(a,b)$ modulo $N_s$ such that no $p\le p_s$ divides both $a$ and $b$.  Note that $(a,b)$ cannot both be even.  As we have seen, up to the parities of $a$ and $b$, for large $s$, the distribution of $X_i$ is approximately uniform, and the distribution of ordered pairs $(X_1,X_2)$ is also approximately uniform.  For each odd prime $p$, the fraction of the pairs $(a,b)$ such that $a,b$ are both divisible by $p$ is approximately $\frac{1}{p^2}$.  Thus, considering all primes, approximately $\prod\limits_{3 \le p \le p_s}(1-\frac1{p^2})$ of the possible pairs do not have a common odd prime factor $\le p_s$. More precisely,
$$\left|\frac{P_{2s}(\text{no odd prime }p\le p_s\text{ divides both }X_1,X_2)}{P_{2s}}- \prod\limits_{3 \le p \le p_s}\left(1-\frac1{p^2}\right) \right| < N_s^2\alpha e^{\frac{-\beta\sqrt{s}}{\ln(s)^2}}.$$

Finally, considering $p = 2$, the probability that both $a,b$ are even is approximately $\frac49$. Thus, we have that $\frac{P_{2s}(C_{s,1}^c)}{P_{2s}}$, the probability that no prime $p\le p_s$ divides both $X_1$ and $X_2$, satisfies
$$\left|\frac{P_{2s}(C_{s,1}^c)}{P_{2s}}-\frac59\cdot\prod\limits_{3 \le p \le p_s}\left(1-\frac1{p^2}\right)\right| < N_s^2\alpha e^{\frac{-\beta\sqrt{s}}{\ln(s)^2}} \le (\ln s)^2\alpha e^{\frac{-\beta\sqrt{s}}{\ln(s)^2}} .$$
Note that the right hand side of the inequality goes to 0 as $s\to \infty$. Thus,
$$\lim\limits_{s\to\infty} \frac{P_{2s}(C_{s,1}^c)}{P_{2s}}=  \frac{5}{9}\cdot\prod_{3\le p}\left(1 - \frac{1}{p^2}\right).$$  

The proof of (2) is similar.  
\end{proof}                                                                                                                                                                                                                                                                                                                                                                                                                                                                                                                                                                                                                                                                                                                                                                                                                                                                                                                                                                                                                                                                                                                                                                                                                                                                                                                                                                                                                                                                                                                                                                                                                                                                                                                                                                                                                                                                                                                                                                                                                                                                                                                                                                                                                                                                                                                                                                                                                                                                                                                                                                                                                                                                                                                                                                         

\subsection{Abelian groups}\label{abelian groups}

Let $V$ be the variety of abelian groups.  To axiomatize $V$, we add to the group axioms the sentence $(\forall x)(\forall y) x+y = y+x$.     

\subsubsection{Elementary invariants}

Szmielew \cite{Szmielew} carried out an elimination of quantifiers for abelian groups, and she gave elementary invariants.  Later, Eklof and Fisher \cite{EF} used saturation to give elementary invariants for modules.  Their methods also yield the Szmielew invariants for abelian groups.  We give invariants for abelian groups below.  For a prime $p$, we write $G[p]$ for the set $\{x\in G:px = 0\}$, which consists of the identity and the elements of order $p$.       

\begin{enumerate}

\item  $\alpha(p,n,k)$, saying $|p^nG|\geq k$, 

\item  $\beta(p,n,k)$, saying $dim(p^nG/_{p^{n+1}G})\geq k$, 

\item  $\gamma(p,n,k)$, saying $dim(p^n G[p])\geq k$, 

\item  $\delta(p,n,k)$, saying $dim(p^nG[p]/_{p^{n+1}G[p]})\geq k$.

\end{enumerate} 

We consider presentations with a single generator $a$ and a single relator.  The free abelian group on one generator is $\Z$, and the other abelian groups on one generator are the finite cyclic groups $C_m$.  We focus on the sentences $\beta(p,n,1)$, which say that there is an element divisible by $p^n$ and not by $p^{n+1}$.  We will see that these sentences are true in $\Z$ and do not have limiting density $0$ or $1$.  For $p = 2$, the limiting density does not exist, while for odd primes $p$, the limiting density exists and has a value strictly between $0$ and $1$.       

\begin{lem}\label{beta-condition}\

\begin{enumerate}

\item  $\beta(p,n,1)$ is true in $\Z$.

\item  $\beta(p,n,1)$ is true in $C_m$ if and only if $p^{n+1}|m$.  

\end{enumerate}

\end{lem}

\begin{proof}

For (1), we note that in $\Z$, the element $p^n$ witnesses the truth of $\beta(p,n,1)$.  For (2), consider $C_m$.  For some $r$ (possibly $0$) and some $m'$ relatively prime to $p$, we have $m = p^r\cdot m'$, and then $C_m\cong C_{p^r}\oplus C_{m'}$.  If $r > n$, then $C_{p^r}$ has an element divisible by $p^n$ and not by $p^{n+1}$, and otherwise, there is no such element.  Furthermore, all elements of $C_{m'}$ are divisible by all powers of $p$.  So, $C_m$ has elements divisible by $p^n$ if and only if $r > n$.    
\end{proof}
  
A relator of length $m$ has the form $w(a) = \sum_{i\leq m} d_i a$, where $d_i = \pm 1$.  We consider the relator of length $0$, representing the empty sum, to be $0$.    

\begin{lem}

$P_s = 1 + 2 + \ldots + 2^s = 2^{s+1} - 1$. 

\end{lem}

\begin{proof}

There is one relator of length $0$.  For $m\geq 1$, there are $2^m$ possible relators $t(a) = d_1a + \ldots + d_m a$, $d_i = \pm 1$.  Then we have \\
$P_s = 1 + 2 + 2^2 + \ldots + 2^s = \frac{1 - 2^{s+1}}{1 - 2} = 2^{s+1} - 1$.
\end{proof} 

We consider the limiting density for $\beta(p,n,1)$ for various combinations of $p$ and $n$.  Recall that the sentence $\beta(p,n,1)$ is true in $\Z$, and it is true in $C_m$ provided that $p^{n+1}|m$.  Let $A$ be the set of relators that give $\Z$, and for $r < p^{n+1}$, let $B_r$ be the set of relators that give $\Z_m$ for the various $m\equiv_{p^{n+1}} r$.  We write $P_s(A)$ and $P_s(B_r)$ for the number of relators of length at most $s$ in the sets $A$ and $B_r$.    

\begin{lem}\

\begin{enumerate}

\item  $P_s(A) = 1+\sum\limits_{0<2m\leq s}\left(\begin{array}{cc}
2m\\m\end{array}\right)$,

\item $A$ has limiting density $0$.

\end{enumerate}      

\end{lem}

\begin{proof}

For (1), note that the relator $t(a)$ gives $\Z$ just in case $t(a)$ has even length $2m$ for some $m$, and $a$ and $-a$ each occur $m$ times.  For (2), look back at Section \ref{early-examples} where we encountered the same quantities as the current $P_s$ and $P_s(A)$.  There, we saw that $\lim_{s\rightarrow\infty}\frac{P_s(A)}{P_s} = 0$.   
\end{proof}

Recall that for a relator $t(a)$, $X$ is the difference between the number of occurrences of $a$ and the number of occurrences of $a^{-1}$ in the term $t(a)$.  For a relator of even length $2m$, $X$ takes only even values $0,\pm 2, \ldots,\pm 2m$.  For an identity of odd length $2m+1$, $X$ takes only odd values
$\pm 1$, $\pm 3$, $\ldots$, $\pm (2m+1)$. Using arguments similar to those in Section \ref{bij-two-id}, we will show that the limiting density of $2^{n+1}|X$ does not exist, and for odd primes $p$, the limiting density of $p^{n+1}|X$ is $\frac{1}{p^{n+1}}$.    

\begin{lem}

For $p = 2$ and $n \geq 0$, the limiting density of $2^{n+1}|X$ does not exist. In particular, 
$\frac{P_{2s}(2^{n+1}|X)}{P_{2s}}\rightarrow (\frac{2}{3})(\frac{1}{2^{n}})$, and $\frac{P_{2s+1}(2^{n+1}|X)}{p_{2s+1}}\rightarrow (\frac{1}{3})(\frac{1}{2^{n}})$.  

\end{lem}

\begin{proof}

We begin with the case where $n = 0$.  Here the calculations are straightforward.  We have $P_{=2m}(2|X) = 2^{2m}$, and $P_{=2m+1}(2|X) = 0$.  Then $P_{2s}(2|X) = 1 + 2^2 + 2^4 + \ldots + 2^{2s} = 
1 + 4 + 4^2 + \ldots + 4^s = 
\frac{4^{s+1} - 1}{3}$.  Therefore, 
$\frac{P_{2s}(2|X)}{P_{2s+1}} = \frac{(\frac{4^{s+1} - 1}{3})}{2^{2s+2} - 1}
\rightarrow \frac{2}{3}$.  
Since $P_{2s+1}(2|X) = P_{2s}$, we have $\frac{P_{2s+1}(2|X)}{P_{2s+1}} = \frac{(\frac{4^{s+1} - 1}{3})}{2^{2s+2} - 1}\rightarrow \frac{1}{3}$.  What we have shown is that if $E$ is the set of relators of even length, then $\frac{P_{2s}(E)}{P_{2s}} \rightarrow \frac{2}{3}$ and $\frac{P_{2s+1}(E)}{P_{2s+1}}\rightarrow \frac{1}{3}$.       

For $n \ge 1$, we again use Theorem \ref{SC}. For every presentation in $E^c$, we have $2^{n+1} \nmid X$, so we may condition on $E$ when we consider even length $2m$.  In the relation $w(a) = \sum_{i\leq 2m} d_i a$, we can consider the sums $d_1+d_2, d_3+d_4, \dots, d_{2m-1}+d_{2m}$. This gives us an $m$-step random walk on $\Z_{2^{n+1}}$ with each step being $2$ with probability $\frac14$, $-2$ with probability $\frac14$, and $0$ with probability $\frac12$.  Dividing everything by 2, we get a random walk with support $\{1,0,-1\}$ on $\Z_{2^n}$.  Thus, Theorem \ref{SC} applies. We have $2^{n+1}\mid X$ exactly when the random walk ends at $0\in \Z_{2^n}$.  The probability of this is $\frac{P_{=2m}(2^{n+1}|X)}{P_{=2m}}\to 1/2^n$. 

Now, as in the proof of Theorem \ref{two-id-thm}, we have that most identities of length at most $s$ will have length at least $\ge \sqrt{s}$.  Since the rate of convergence in Theorem \ref{SC} is exponential and all identities in $E$ have even length, we can pass from the probability for identities of a fixed length to the probability for identities with length $\le s$, and we get $\frac{P_{2s}(2^{n+1}|X \mid E)}{P_{2s}(E)}\to 1/2^n$.

Since $2^{n+1}\mid X$ only when $X$ is even, i.e., the identity is in $E$, and the above probability $1/2^n$ was conditioned to $E$, we have the desired 
$$\frac{P_{2s}(2^{n+1}|X)}{P_{2s}} = \frac{P_{2s}(2^{n+1}|X \mid E)}{P_{2s}} = \frac{P_{2s}(2^{n+1}|X \mid E)}{P_{2s}(E)} \cdot \frac{P_{2s}(E)}{P_{2s}} \to (\frac{1}{2^n})\cdot (\frac23).$$
The odd case can be proved similarly.        
\end{proof}

\begin{lem}\label{beta-prob}

For odd primes $p$, $p^{n+1}|X$ has limiting density $\frac{1}{p^{n+1}}$.  

\end{lem}

\begin{proof}

For a fixed even length $2m$, we get a random walk on $\Z_{p^{n+1}}$ supported on $\{2, 0, -2\}$---a single step increases the length by $2$.  By Theorem \ref{SC}, we have $\frac{P_{2m}(p^{n+1}\mid X)}{P_{2n}} \to \frac{1}{p^{n+1}}$.  This random walk converges to the uniform distribution for even lengths, and the same is true for the odd lengths.  

As before, we split the set of relators of length at most $s$ into two parts, those of length less than $\sqrt{s}$, and those of length at least $\sqrt{s}$.  Let $S_s$ be the number of relators of length at most $s$ for which the length is less than $\sqrt{s}$, and let $L_s$ be the number for which the length is at least $\sqrt{s}$.  Then $\frac{S_s}{P_s}\rightarrow 0$, so $\frac{L_s}{P_s}\rightarrow 1$.  
Then the exponential rate of convergence of Theorem \ref{SC} gives 
\[\frac{P_s(p^{n+1}|X)}{P_s}\rightarrow\frac{1}{p^{n+1}}.\]      
\end{proof}

\section{Generalizing}
\label{general}

In this section, we give general conditions that imply some of the behaviors that we saw in Section \ref{illu-ex}.  Our languages will have finitely many unary function symbols, and we may also allow finitely many constants.  

\subsection{Generalized bijective varieties} 

In Section \ref{early-examples}, we considered the variety with axioms saying of a pair of unary function symbols $S,S^{-1}$ that they are inverses, and we showed that for presentations with a single generator $a$ and a single identity of the form $t(a) = a$, the sentences true in the free structure are exactly those with limiting density $1$.  In this subsection and the next, we turn our attention to varieties of structures with multiple bijective unary functions, possibly with additional axioms.  We might suppose that the language of has unary function symbols $g_1, g_1^{-1}, \cdots, g_n, g_n^{-1}$, and that our varieties have axioms saying that for each $i$, $g_i$ and $g_i^{-1}$ are inverses.  However, the assumption that the functions have inverses named by function symbols turns out to be unnecessary once we know that the functions are $1-1$ and onto.  

\begin{defn}\label{def-bij}

Let $L$ be a language with unary function symbols $f_1, \ldots, f_n$, and let $V$ be an algebraic variety with theory $T$.  The variety is \emph{generalized bijective} if for all $i$,
$T\vdash (\forall x,y)(f_i(x) = f_i(y) \rightarrow x=y)$ and \\
$T\vdash (\forall y) (\exists x)\ f_i(x) = y$. 

\end{defn}

The result below says that for a generalized bijective variety, the basic functions have inverses named by terms.  

\begin{prop}\label{bij-inverse}

Let $T$ be the theory of a generalized bijective variety in the language $\{f_1,\ldots,f_n\}$.  Then for each $f_i$, there is some word $u_i$ such that \\
$T\vdash(\forall x) f_i\circ u_i(x) = u_i \circ f_i(x) = x$.

\end{prop}

\begin{proof}

Fix $i$, and let $F$ be the free structure on one generator $a$.  There is some $b\in F$ with $f_i(b) = a$.  We can express $b$ as $u_i(a)$ for some word $u_i$. Then $F\models f_i\circ u_i(a) = a$.  Recall that in a variety, if an atomic formula is true of the generating tuplea in the free structure, then it holds on all tuples in all structures \cite[Theorem 11.4]{BS}.  Thus, $T \vdash( \forall x)\ f_i\circ u_i(x) = x$.  In $F$, let $x = f_i(a)$.  We have $f_i\circ u_i \circ f_i (a) = f_i(a)$.  Since $f_i$ is injective, this means that $F\models u_i \circ f_i(a) = a$.  Hence, $T \models (\forall x)\ u_i \circ f_i(x) = x$.  This completes the proof.
\end{proof}

\begin{defn}

Let $V$ be a variety in the language $\{f_1,\ldots,f_n\}$, where each $f_i$ is unary.  The variety is \emph{commutative} if the axioms imply $(\forall x) f_i(f_j(x)) = f_j(f_i(x))$ for all $i,j$.

\end{defn}

Our main general result, Theorem \ref{bij-str-thm}, says that for a commutative generalized bijective variety $V$ and presentations with a single generator and a single identity, the zero--one law holds.  Moreover, the sentences with density $1$ are those true in the free structure.  To prove Theorem \ref{bij-str-thm}, we will use a version of Gaifman's Locality Theorem, which we discuss below.  

\subsection{Gaifman's Locality Theorem}\label{Gaif-sec}

We state a special version of Gaifman's Locality Theorem for generalized bijective varieties, and we sketch a proof using saturation.  Fix a language $L$ consisting of unary function symbols $f_1,\ldots,f_m$.  Below, we define the \emph{Gaifman graph} of an $L$-structure.  Gaifman defined the graph for structures in a finite relational language.  When convenient, we treat the unary functions as binary relations.  

\begin{defn}

Let $\mathcal{A}$ be an $L$-structure.  The \emph{Gaifman graph} of $\mathcal{A}$ is the undirected graph with universe equal to that of $\mathcal{A}$, and with an edge between $x$ and $y$ if and only if $f_i(x) = y$ or $f_i(y) = x$ for some $i$.    

\end{defn}

We define an equivalence relation $\sim$ on $\mathcal{A}$ such that $x\sim y$ if $x$ and $y$ belong to the same connected component in the Gaifman graph; i.e., there is a finite path leading from $x$ to $y$.  

\begin{defn} [distance, $d(x,y)$]

For $x,y\in\mathcal{A}$, the \emph{distance} between $x$ and $y$ is the least $r$ such that there is a path of length $r$ from $x$ to $y$.  We write $d(x,y) \geq r$, $d(x,y) > r$ to indicate that the distance is, respectively, at least $r$, or greater than $r$.  

\end{defn}

\noindent
\textbf{Remark}:  Elements $x,y$ lie in different connected components just in case $d(x,y) > r$ for all $r$.  

\bigskip

We consider substructures of $\mathcal{A}$.  Note that two connected components, thought of as substructures, are isomorphic if there is a map from one onto the other that preserves the unary functions $f_i$, which we think of as binary relations.  The structure $\mathcal{A}$ is determined, up to isomorphism, by the number of connected components of different isomorphism types.

\begin{defn} [$r$-ball, $B_r(a)$, $B_r(\bar{a})$]

Let $\mathcal{A}$ be a structure and let $r\in\omega$.  

\begin{enumerate}

\item  For $a\in\mathcal{A}$, the \emph{$r$-ball around $a$} is $B_r(a) = \{x\in\mathcal{A}:d(a,x)\leq r\}$.  

\item  For $\bar{a}\in\mathcal{A}^n$, we write $B_r(\bar{a})$ for the set 
$\cup_{i < n}B_r(a_i)$. 

\item  We write $B_\infty(a)$ for the connected component of $a$, or $\cup_r B_r(a)$.  

\item  We write $B_\infty(\bar{a})$ for the union of the connected components of elements of $\bar{a}$, or $\cup_i B_\infty(a_i)$.  

\end{enumerate}

\end{defn}   

Let $V$ be a generalized bijective variety for the language $L$.  
For $\mathcal{A}\in V$, each element has a unique image and a unique pre-image under each $f_i$.  We show that for each $r$ and $n$, there is a finite set of formulas $\alpha(\bar{x})$ that describe, for all $\mathcal{A}\in V$, the possible substructures $B_r(\bar{a})$ for $n$-tuples $\bar{a}$.  

\begin{lem}
\label{Lem4.7}

Let $V$ be a generalized bijective variety for the language $L$. For each $r$ and $n$, there is a finite set $C_{r,n}$ of formulas $\alpha(\bar{x})$, such that 

\begin{enumerate}

\item  for each $\mathcal{A}\in V$, each $n$-tuple $\bar{a}$ in $\mathcal{A}$ satisfies a unique formula 
$\alpha(\bar{x})\in C_{r,n}$,

\item  for all $\mathcal{A},\mathcal{A}'\in V$, if $n$-tuples $\bar{a}$ in $\mathcal{A}$ and $\bar{a}'$ in $\mathcal{A}'$  satisfy the same formula $\alpha(\bar{x})\in C_{r,n}$, then there is an isomorphism from $B_r(\bar{a})$ onto $B_r(\bar{a}')$ that takes $\bar{a}$ to $\bar{a}'$.  

\end{enumerate} 
Moreover, we may take the formulas $\alpha(\bar{x})$ in $C_{r,n}$ to be existential.  We may equally well take them to be universal.     

\end{lem}

\begin{proof}

We describe the \emph{possible elements} of $B_r(\bar{x})$ inductively as follows.  The set $B_{0}(\bar{x})$ has just the members of the $n$-tuple $\bar{x}$ as possible elements. Now, suppose we have the possible elements of $B_r(\bar{x})$ for some $r\geq 0$. We will set the possible elements of $B_{r+1}(\bar{x})$ to be the elements of $B_{r}(\bar{x})$ together with additional possible elements $z$ obtained as follows:  Take some $y\in B_r(\bar{x})$ corresponding to a node at a distance $r$ from some $x\in \bar{x}$ and follow an arrow labeled $f_i$ or $f_i^{-1}$ from $y$ to $z$; please note that $f_i^{-1}$ is shorthand for the term that acts as an inverse to $f_i$ from Proposition \ref{bij-inverse}.

We may think of the possible elements of $B_r(\bar{x})$ as terms $u(x)$, where $u$ is a string of $f_i$, $f_i^{-1}$ of length at most $r$.   For an actual structure in our generalized bijective variety, with an actual tuple $\bar{a}$ corresponding to $\bar{x}$, we may have equalities---different paths may lead to the same point.  For $\mathcal{A}\in V$ generated by $\bar{a}$, the elements of $B_r(\bar{a})$ are equivalence classes of terms $u(a_i)$, where $u$ is a string of $f_i,f_i^{-1}$ of length at most $r$.  We have an existential formula saying that there exist $y$'s corresponding to the possible elements of $B_r(\bar{x})$ such that the structure has a specific atomic diagram.  We also have a universal formula saying that for all $y$'s corresponding to the possible elements of $B_r(\bar{x})$, the structure has a specific atomic diagram.           
\end{proof}

We fix sets of formulas $C_{r,n}$ as in the lemma.  Gaifman's Locality Theorem says that any formula $\varphi(\bar{x})$ (in a relational language) can be expressed as a finite Boolean combination of ``local'' formulas and ``local'' sentences (see the references \cite{Libkin}, \cite{Gaifman}, \cite{Kolaitis}).  For our setting, we take the local formulas and local sentences to be as follows. 

\begin{defn}\

\begin{enumerate}

\item  The \emph{$r$-local formulas} $\bar{x}$ are those in $C_{r,n}$ for various $n$.  

\item  The \emph{$r$-local sentences} have one of the following forms:

\begin{enumerate}

\item  $(\exists v_1, \cdots, v_s) \left( \bigwedge\limits_{i} \alpha_i(v_i)\ \& \bigwedge\limits_{i<j}d^{>2r}(v_i,v_j)\right)$,\\ for some $s$ and $\alpha_i(x)\in C_{r,1}$,

\item  $(\exists v)\alpha(v)$, for some $\alpha\in C_{r,1}$.  

\end{enumerate}

\end{enumerate}

\end{defn}

\begin{remark}

This definition is similar to Gaifman's, except that we allow only special formulas in $C_{r,n}$.  Note that the formulas in $C_{r,n}$ already give information on whether the distance between $x_i$ and $x_j$ is greater than $2r$.  Indeed, if $d(x_i,x_j)\leq 2r$, the formula will contain a conjunct that says (in the rational language) $t(x_i) = t'(x_j)$ for some $t,t'$ of length at most $r$.  Thus, we may equivalently replace 2(a) by $(\exists v_1, \cdots, v_s )\alpha(v_1,\cdots, v_s)$ for some $\alpha\in C_{r,s}$.  We chose the form above to stay closer to Gaifman's definition. 

\end{remark}

\begin{defn}

A formula or sentence is \emph{local} if it is $r$-local for some $r$. 

\end{defn}  

Here is our special version of Gaifman's Locality Theorem, where the local formulas and sentences are as defined above. 

\begin{thm}\label{Gaifman}

Let $V$ be a generalized bijective variety with theory $T$.   

\begin{enumerate}

\item  Any elementary first order sentence $\varphi$ is equivalent over $T$ to a sentence $\varphi^*$ that is a finite Boolean combination of local sentences.      

\item  Any elementary first order formula $\varphi(\bar{x})$ with free variables $\bar{x}$ is equivalent over $T$ to a formula $\varphi^*(\bar{x})$ that is finite Boolean combination of local sentences and local formulas.  In fact, we may take $\varphi^*(\bar{x})$ to be a finite disjunction of formulas $\alpha_i(\bar{x})\ \&\ \beta$, where for each $i$, $\alpha_i(\bar{x})$ is a single local formula, and 
$\beta$ is a finite conjunction of local sentences and negations of local sentences.     

\end{enumerate}   
\end{thm}

We sketch a proof using saturation.  We begin with some definitions and lemmas.  
 
\begin{defn}\

\begin{itemize}

\item  For $\mathcal{A}\in V$, the \emph{local theory} of $\mathcal{A}$ is the set of all local sentences and negations of local sentences that are true in $\mathcal{A}$.  

\item  For $\bar{a}$ in $\mathcal{A}$, the \emph{local type} of $\bar{a}$ is the set of formulas generated by the local theory and the set of local formulas true of $\bar{a}$ in $\mathcal{A}$. 

\end{itemize} 

\end{defn}

Note that for $\bar{a}$ in $\mathcal{A}$ and $\bar{a}'$ in $\mathcal{A}'$ of the same length, if the local type of $\bar{a}$ in $\mathcal{A}$ is contained in the local type of $\bar{a}'$ in $\mathcal{A}'$, then the local types are the same.     

\begin{lem}
\label{Lem4.9}

Let $\mathcal{A},\mathcal{A}'\in V$.  If $n$-tuples $\bar{a}$ in $\mathcal{A}$ and $\bar{a'}$ in 
$\mathcal{A}'$ satisfy the same local type, then there is a partial isomorphism $f$ from $B_\infty(\bar{a})$ onto $B_\infty(\bar{a}')$ such that $f(\bar{a}) = \bar{a}'$.         

\end{lem}

\begin{proof} 
  
The fact that the tuples $\bar{a}$ and $\bar{a}'$ satisfy the same local type means that the structures $\mathcal{A}$ and $\mathcal{A}'$ satisfy the same local theory, and for each $r$, the tuples $\bar{a},\bar{a}'$ satisfy the same unique formula $\alpha(\bar{x})\in C_{r,n}$.  By Lemma \ref{Lem4.7}, for each $r$, there is an isomorphism $p$ from $B_r(\bar{a})$ onto $B_r(\bar{a}')$ taking $\bar{a}$ to $\bar{a}'$.  We have a tree of these finite partial isomorphisms $p$ between $B_\infty(\bar{a})\subseteq\mathcal{A}$ and 
$B_\infty(\bar{a}')\subseteq\mathcal{A}'$, where at level $r$, we put the isomorphisms from $B_r(\bar{a})$ onto $B_r(\bar{a}')$ that take $\bar{a}$ to $\bar{a}'$, and at level $r+1$, the successors of a given partial isomorphism $p$ from level $n$ are the extensions of $p$ taking $B_{r+1}(\bar{a})$ isomorphically onto $B_{r+1}(\bar{a}')$.  If $B_\infty(\bar{a})$ is infinite, then the tree is infinite, and it is finitely branching, so by K\"{o}nig's Lemma, there is a path $(p_r)_{r\in\omega}$.  The desired isomorphism is $\cup_r p_r$.  If the substructure $B_\infty(\bar{a})$ is finite, then it is contained in $B_r(\bar{a})$ for some $r$, and $p_r$ is the desired isomorphism.
\end{proof} 

For any $\mathcal{A}\in V$, the isomorphism type of $\mathcal{A}$ is determined by the number of connected components of each isomorphism type.  Suppose $\mathcal{A}$ is saturated, of infinite cardinality $\kappa$.  In $\mathcal{A}$, a local type $\Gamma(\bar{x})$ is satisfied if it is finitely satisfied.  For a local type $\Gamma(x) = \{\alpha_r(x):r\in\omega\}$, there are at least $n$ realizations of $\Gamma(x)$ on different connected components if and only if for all $r$, $\mathcal{A}$ satisfies the $r$-local sentence saying that there are at least $n$ elements satisfying $\alpha_r(x)$ and at a distance greater than $2r$.  The number of connected components with an element satisfying $\Gamma(x)$ is either finite or $\kappa$.  This yields the following.          

\begin{lem} 
\label{Lem4.10}

Suppose $\mathcal{A},\mathcal{A}'\in V$ are saturated and of the same cardinality $\kappa$.  If $\mathcal{A},\mathcal{A}'$ satisfy the same local sentences, then $\mathcal{A}\cong\mathcal{A}'$.

\end{lem}

\begin{proof}

Since $\mathcal{A},\mathcal{A}'$ are saturated, of the same cardinality, and satisfy the same local sentences, they realize the same local types, and they have the same number of connected components of each isomorphism type.  Hence, they are isomorphic.
\end{proof}

Knowing what the saturated structures in the variety $V$ look like, we see that for any countable $\mathcal{A}\in V$, there exists a saturated structure $\mathcal{A}^*$ of cardinality $2^{\aleph_0}$ such that $\mathcal{A},\mathcal{A}^*$ satisfy the same local sentences.     

\begin{lem}

If $\mathcal{A},\mathcal{A}'\in V$ have the same local theory, then they are elementarily equivalent.

\end{lem} 

\begin{proof}

Let $\mathcal{A}^*$ and $(\mathcal{A}')^*$ be saturated models of the common local theory of $\mathcal{A},\mathcal{A}'$ such that $\mathcal{A}^*,(\mathcal{A}')^*$ both have cardinality $2^{\aleph_0}$.  Applying Lemma \ref{Lem4.10}, we see that $\mathcal{A}^*\cong(\mathcal{A}')^*$.  Hence, $\mathcal{A},\mathcal{A}'$ are elementarily equivalent.
\end{proof} 

\begin{lem}
\label{Lem4.11}

Take $n$-tuples $\bar{a},\bar{a}'$ in $\mathcal{A}$.  If $\bar{a},\bar{a}'$ satisfy the same local type, then there is an automorphism of $\mathcal{A}$ that takes $\bar{a}$ to $\bar{a}'$.    

\end{lem} 

\begin{proof}

We have a partial isomorphism $f$ from $B_\infty(\bar{a})$ onto $B_\infty(\bar{a}')$ such that $f(\bar{a}) = \bar{a}'$.  This extends to an automorphism that agrees with $f$ on $B_\infty(\bar{a})$, with $f^{-1}$ on $B_\infty(\bar{a}')$, and with the identity on the rest of $\mathcal{A}$.  
\end{proof}   

\begin{lem}
\label{Lem4.12}

If $\mathcal{A}\models\varphi$, then there is a sentence $\psi$ true in $\mathcal{A}$ such that $\psi$ is a finite conjunction of local sentences and negations of local sentences and $T\vdash (\psi\rightarrow\varphi)$. 

\end{lem} 

\begin{proof}

If $S$ is the local theory of $\mathcal{A}$, then $T\cup S\vdash\varphi$.  Then there is some $\psi$, the conjunction of a finite subset of $S$, such that $T\vdash (\psi\rightarrow\varphi)$.  
\end{proof}

For a formula $\varphi(\bar{x})$ with an $n$-tuple $\bar{x}$ of variables, we have the following.

\begin{lem}
\label{Lem4.13}

If $\mathcal{A}\models\varphi(\bar{a})$, then there is a formula $\psi(\bar{x}) = (\alpha(\bar{x})\ \&\ \beta)$ such that $\beta$ is a finite conjunction of sentences in the local theory of $\mathcal{A}$, $\alpha(\bar{x})$ is a local formula satisfied by $\bar{a}$ in $\mathcal{A}$, and $T\vdash(\forall\bar{x})(\psi(\bar{x})\rightarrow\varphi(\bar{x}))$. 

\end{lem}

\begin{proof}

We have a saturated model $\mathcal{B}$ of cardinality $2^{\aleph_0}$ with a tuple $\bar{b}$ satisfying the type of $\bar{a}$.  If $\mathcal{B}'$ is saturated and satisfies the local theory of $\mathcal{A}$ and $\mathcal{B}$, there is an isomorphism $f$ from $\mathcal{B}$ onto $\mathcal{B}'$.  If $\bar{b}'$ is an $n$-tuple in $\mathcal{B}'$ satisfying the local type of $\bar{a}$ and $\bar{b}$, we may suppose that $f(\bar{b}) = \bar{b}'$.  Hence, $\bar{b}'$ realizes the complete type of $\bar{a}$.  This shows that the local theory of $\mathcal{A}$ and the local type of $\bar{a}$ generate the full theory and type.  If $\chi(\bar{x})$ is a finite conjunction of local formulas and negations of local formulas in the local type of $\bar{a}$, then there is a single formula $\alpha(\bar{x})$ in the local type of $\bar{a}$ that implies $\chi(\bar{x})$---take $\alpha(\bar{x})\in C_{r,n}$ for sufficiently large~$r$.  
\end{proof}  

A standard model-theoretic argument gives the following. 

\begin{prop} 

Any elementary first order sentence $\varphi$ is equivalent over $T$ to a finite disjunction of local sentences and negations of such sentences.     

\end{prop}

\begin{proof}  

For each $\mathcal{A}\in V$ satisfying $\varphi$, choose $\psi$ as in Lemma \ref{Lem4.12}, a finite conjunction of local sentences and negations, true in $\mathcal{A}$, such that $T\vdash (\psi\rightarrow\varphi)$.  Let $S$ be the set of chosen sentences.  Now, $T\cup\{\neg{\psi}:\psi\in S\}\cup\{\varphi\}$ is inconsistent, so there is a finite set $S'\subseteq S$ such that $T\vdash (\varphi\rightarrow \bigvee_{\psi\in S'} \psi)$.  Then $\varphi$ is equivalent over $T$ to the disjunction of the sentences in $S'$. 
\end{proof} 

Here is the companion result for formulas with free variables.  

\begin{prop}

For any formula $\varphi(\bar{x})$ with free variables among $\bar{x}$, there is a formula $\varphi^*(\bar{x})$ equivalent over $T$ to $\varphi(\bar{x})$ such that $\varphi^*(\bar{x})$ is a finite disjunction of formulas $(\alpha(\bar{x})\ \&\ \beta)$, where $\beta$ is a conjunction of local sentences and negations and $\alpha(\bar{x})$ is a local formula.  

\end{prop} 

\begin{proof}

We replace $\bar{x}$ with a tuple of constants $\bar{c}$.  For each model $\mathcal{A}$ of $T$ and each tuple $\bar{a}$ satisfying $\varphi(\bar{x})$, choose a formula $\psi(\bar{x})$ in the local type of $\bar{a}$ such that $T\vdash (\psi(\bar{c})\rightarrow\varphi(\bar{c}))$.  Let $S$ be the set of chosen formulas.  Now, $T\cup\{\neg{\psi(\bar{c})}:\psi(\bar{c})\in S\}\cup\{\varphi(\bar{c})\}$ is inconsistent, so for some finite $S'\subseteq S$, $T\vdash(\varphi(\bar{c})\rightarrow\bigvee_{\psi(\bar{c})\in S'}\psi(\bar{c}))$.  We may take $\psi$ of the form $\alpha(\bar{x})\ \&\ \beta$, where $\beta$ is the conjunction of the local sentences in $S'$ and $\alpha(\bar{x})$ is the local formula in $C_{r,n}$ that is true of $\bar{a}$, where $r$ is greatest such that $S'$ contains a formula in~$C_{r,n}$.           
\end{proof} 

\begin{remark}

For our special version of Gaifman's Locality Theorem, the local formulas may be taken to be either existential or universal.  Thus, over a completion of $T$ (or over the set of local sentences in the complete theory), each formula is equivalent to an existential formula, and to a universal formula.  

\end{remark}

\subsection{The group associated to a generalized bijective variety}\label{subsec:group-gen-bij-thy}

Let $V$ be a generalized bijective variety with theory $T$.  There is an equivalence relation on strings of function symbols such that strings $t,t'$ are equivalent if $T\vdash(\forall x) t(x) = t'(x)$.  For a string of symbols $t$, we may write $\len(t)$ for the length of $t$.  We will associate to the variety $V$ a group $G(V)$, whose elements are the equivalence classes of strings.  

\begin{defn}  [Gaifman group, $G(V)$]

For a generalized bijective variety $V$, the \emph{Gaifman group} is the group $G(V)$ consisting of equivalence classes of strings of symbols under the operation induced by concatenation of strings.  

\end{defn} 

The identity in $G(V)$ is the equivalence class of the empty string.  For each function symbol $f_i$, we fix a term $u_i$ that names the inverse, as in Proposition~\ref{bij-inverse}.  We may write $f_i^{-1}$ for $u_i$.  The inverse function extends in a natural way to any word $v$ in $f_1, \cdots, f_n$.  
Let $F$ be the element of $V$ obtained as the free structure generated by the finite tuple $\bar{a}$.  The group $G(V)$ has a natural action on $F$, taking $t\in G(V)$ and $b\in F$ to $t(b)$.  Since $b = t'(a)$ for some $t'$, the action takes $t'(a)$ to $t\circ t'(a)$.      

\begin{defn} [orbit under action of $G(V)$]

For $\mathcal{A}\in V$ and $b\in \mathcal{A}$, the \emph{orbit} of $b$ under the action of $G(V)$ is the set of all $x$ such that for some $t\in G(V)$, $t(b) = x$. 

\end{defn}

\begin{note}

For $\mathcal{A}\in V$ and $b\in\mathcal{A}$, the orbit of $b$ under the action of $G(V)$ is just the set of elements of $\mathcal{A}$ generated by $b$.  The \emph{automorphism orbit} of $b$ results from the action of the group of automorphisms.     
  
\end{note}

\begin{lem}

Let $V$ be a generalized bijective variety, and let $F$ be the free structure in $V$ generated by the tuple $\bar{a}$. The action of $G(V)$ on $F$ is well defined and simply transitive on the orbits.   

\end{lem} 

\begin{proof}

We first prove that the action is well defined. Suppose $t_1 = t_2$ in $G(V)$. Without loss of generality, we assume that $t_1$ is obtained from $t_2$ by applying an identity of $G(V)$, say 
$w = w'$.  This means that $t_1 = uw(w')^{-1}v$ and $t_2 = uv$ for some words $u,v$. Then 
 $(t_1,t'(a)) \mapsto t_1\circ t'(a) = uw(w')^{-1}vt'(a)$.  Since $w = w'$ is an identity in $G(V)$, we have $T\vdash (\forall x) w(x) = (w')(x)$ and for an element $a$ of $F$, $F\models uw(w')^{-1}vt'(a) = uw'(w')^{-1}vt'(a) = uvt'(a) = t_2t'(a)$. Thus, the action is well defined.  

Recall that every element $x$ of $F$ has the form $t(a_i)$ for some generator $a_i$, and every such $x$ is in the orbit of $a_i$.  Thus, every orbit in $F$ has the form $\{t(a_i):t\in G(V)\}$ for some generator $a_i$.
Now, take $x = t(a_i)$ in $F$ and suppose that $F\models u \circ t(a_i) = v \circ t(a_i)$.  Since $F$ is free, we have that $T \models (\forall x) ut(x) = vt(x)$. Therefore, $ut = vt$ holds in the group $G(V)$, so by cancellation, we have $u = v$. Thus, the action is simply transitive on its orbits.  
\end{proof}

For our commutative generalized bijective variety with theory $T$, we have the following.    

\begin{lem}\ 

\begin{enumerate}

\item  For $u,v,w\in G(V)$, $T\vdash (\forall x)(u(w(x)) = v(w(x))\leftrightarrow u(x) = v(x))$.

\item  For $\alpha\in C_{r,1}$, $T\vdash (\forall x)(\alpha(w(x))\leftrightarrow \alpha(x))$.  

\end{enumerate}

\end{lem}

For structures $\mathcal{A}\in V$ with a single generator $a$, all elements have the same local type.  In fact, they are in the same automorphism orbit as well as the same orbit under the action of $G(V)$.  

\begin{lem}

Suppose $\mathcal{A}\in V$ is generated by $a$.  For $\alpha\in C_{r,1}$, \\
$\mathcal{A}\models (\exists x)\alpha(x)\leftrightarrow (\forall x)\alpha(x)$. 

\end{lem}

Consider a local sentence $\rho$ saying that there exists $\bar{x}$ with $x_i$ satisfying $\alpha_i\in C_{r,1}$ and with $d(x_i,x_j) > 2r$ for $i < j$.  For $\mathcal{A}$ generated by a single element $a$, $\rho$ cannot be true unless the $\alpha_i$'s are all the same and $\mathcal{A}$ has a tuple of elements $\bar{x}$ such that $d(x_i,x_j) > 2r$ for $i < j$.  Thus, the important local invariants are the sentences $(\exists x)\alpha(x)$ for $\alpha\in C_{r,1}$ and the sentences saying that there are at least $n$ elements at a distance at least $2r$.  We will show that for these important sentences, the ones true in $F$ have density $1$.

For a string $t$ of function symbols, we write $t^n$ for the $n$-fold concatenation of $t$.  We write $\langle t\rangle$ for the subgroup of $G(V)$ generated by the equivalence class of $t$---the elements are the equivalence classes of the strings $t^n$, $t^{-n} = (t^{-1})^n$.  We need to understand truth in the structure $\mathcal{A}$ with presentation $a|R$, where $R$ is a single identity.  Any identity is equivalent over $T$ to a canonically chosen identity of the form $t^*(a) = a$, where the length of $t^*$ is bounded by a constant multiple of the length of $R$.  The next lemma will tell us a great deal about truth in $\mathcal{A}$.      

\begin{lem}\label{coset} 

Let $V$ be a generalized bijective variety, and consider presentations $a|R$, where $R$ is an identity equivalent to one of the form $t^*(a) = a$.  Then for $u,v\in G(V)$, $\langle a|R\rangle\models u(a) = v(a)$ iff $u,v$ are in the same left coset of $\langle t^* \rangle$.     

\end{lem}

\begin{proof}

Let $\mathcal{A} = \langle a|R\rangle $.  

\bigskip
\noindent
$\Leftarrow$:  Without loss of generality, suppose $v = u(t^*)^n$.  In $\mathcal{A}$, we have 
\[v(a) = u(t^*)^n(a) = u(t^*)^{n-1}(a) = \cdots = u(a)\ .\] 

\noindent
$\Rightarrow$:  Now, suppose $u(a) = v(a)$ in $\mathcal{A}$.  Then $T\cup\{t^*(a) = a\}$ must prove     
\[u(a)= x_0(a) = x_1(a) = \cdots = x_\ell(a) = v(a)\ ,\]  
where for each $i < \ell$, we have one of the following:

\begin{enumerate}

\item [(i)]  $x_{i+1} = x_i t^*$, 

\item [(ii)]  $x_i = x_{i+1}t^*$, or 

\item [(iii)]  $x_i(a) = x_{i+1}(a)$. 

\end{enumerate}
In the first two cases, $x_i$ and $x_{i+1}$ are clearly in the same left coset of $\langle t^* \rangle$.  In the third case, $x_i = x_{i+1}$ in $G(V)$, so again $x_i$ and $x_{i+1}$ are in the same left coset of $\langle t^* \rangle$.
\end{proof}

For a given identity $u(a) = v(a)$, we are interested in the identities $R$ such that $\langle a|R\rangle\models u(a) = v(a)$.  The lemma above lets us recognize these identities.  We come to the theorem that gives conditions under which the sentences true in the free structure have limiting density $1$.    

\begin{thm}
\label{bij-str-thm}

Let $V$ be a commutative generalized bijective variety in the language $\{f_1, \cdots, f_n\}$, and consider presentations with a single generator $a$ and a single identity.  Let $F$ be the free structure on $a$.  If $F$ is  infinite, then the sentences true in $F$ have limiting density $1$.  

\end{thm}

\begin{proof}

We show that for the important sentences $\alpha$, if $\alpha$ is true in $F$, then it has density $1$, and if $\alpha$ is false in $F$, then it has density $0$.  For structures in $V$ with generator $a$, the important sentences say one of the following: 

\begin{enumerate}

\item $(\exists x)\alpha(x)$ for $\alpha\in C_{r,1}$---this is equivalent to a finite conjunction of formulas of the form $u(a) = v(a)$ or $u(a)\not= v(a)$.	 

\item  $(\exists x_1,\ldots x_n)\bigwedge_{i < j} d(x_i,x_j) > 2r$.  

\end{enumerate}  

If $F$ is infinite, then we can show that any sentence of the second form true in $F$ is implied over $T$ by a sentence of the first form true in $F$.  A saturated model of the theory of $F$ has infinitely many connected components, and the sentence $(\exists x_1,\ldots x_n)\bigwedge_{i < j} d(x_i,x_j) > 2r$ is clearly true in this model.  Therefore, it is true in $F$.  Take witnesses $x_1,\ldots,x_n$, where $x_i = w_i(a)$.  Choose $k$ such that all $x_i$ are in $B_k(a)$, and take $\alpha\in C_{1,k}$ true of $a$ in $F$.  Then over $T$, $(\exists x)\alpha(x)$ implies $(\exists x_1,\ldots x_n)\bigwedge_{i < j} d(x_i,x_j) > 2r$. 

The group $G(V)$ is abelian and finitely generated, so it is a finite direct product of cyclic groups generated by some elements $b_1, \cdots, b_k$.  We write $\Pi_i(x)$ for the projection of an element $x$ on the subgroup generated by $b_i$.  Since $G(V)$ is infinite, some $b_i$ must have infinite order.  Without loss of generality, we suppose $b_1$ has infinite order and generates a copy of $\Z$.  We focus on $\Pi_1(x)$, and we suppose that the values are integers.  

Each identity $R$ has the form $t(a) = t'(a)$, but this is equivalent to an identity of the form $t^*(a) = a$.  Let $e_0 = \max_i |\Pi_1(f_i)|$.  If $\len(t)\leq r$, then the projection $\Pi_1(t)$ is an integer bounded by $r\cdot e_0$.  If $\len(t),\len(t')\leq r$, then $d(t,t')\leq 2r$.  Then $|\Pi_1(t) - \Pi_1(t')|\leq 2r\cdot e_0$.  To prove Theorem \ref{bij-str-thm}, it is enough to show that all statements of the form $t(a) = t'(a)$ or $t(a)\not= t'(a)$ true in $F$ have limiting density $1$.  The proof consists of two steps.    

\begin{enumerate}

\item  The first step is to show that for a fixed $k$, the set of presentations\\ 
$a|t(a) = t'(a)$ such that $|\Pi_1(t) - \Pi_1(t')| < k$ has limiting density $0$. 

\item  The second step is to show that for a fixed $k$ and a fixed identity $R$ of the form $t(a) = t'(a)$, if $|\Pi_1(t) - \Pi_1(t'))| > e_0k$, then for any $u$, $v$ such that $d(u,v)\leq k$ in the Gaifman graph $G(F)$, we have $F\models u(a) = v(a)$ if and only if $\langle a \mid R \rangle \models u(a) = v(a)$. 

\end{enumerate}  

Toward the first step, we prove some lemmas.   

\begin{lem}\ \label{number-of-id}

\begin{enumerate}

\item  The number of identities of length $m$ is $n^m(m+1)$. Furthermore, for every $0 \le k \le m+1$, there are exactly $n^m$ identities of length $m$ in which $t$ (the string of function symbols on the left side) has length 
$k$. 

\item  $P_s = \frac{n^{s+1}(s+2)(n-1) + 1}{(n-1)^2}$.

\end{enumerate}

\end{lem}

\begin{proof}

For (1), the number of strings of function symbols of length $m$ is $n^m$.  To determine an identity $t(a) = t'(a)$, we choose one of the $m+1$ initial segments to serve as the left-hand side.
For (2), we simply note that 
\begin{eqnarray*}
P_s & = & \sum_{0\leq m\leq s}(m+1)n^m = (1 + 2n + \ldots + (s+1)n^s) \\
 & = & \displaystyle \frac{(s+2)n^{s+2} - (s+2)n^{s+1} + 1}{(n-1)^2} = \frac{n^{s+1}(s+2)(n-1) + 1}{(n-1)^2}.
\end{eqnarray*}

\end{proof}

The next lemma may by interpreted as saying that a random identity of length $\le s$ has length $ > \sqrt{s}$.

\begin{lem}
\label{sqareroot}

$\lim_{s\rightarrow\infty} \frac{P_{s^2} - P_s}{P_{s^2}} = 1$.

\end{lem}

\begin{proof}

Using Lemma \ref{number-of-id}, we get $\frac{P_s}{P_{s^2}} = 
\frac{n^{s+1}(s+2)(n-1) + 1}{n^{s^2+1}(s^2(n-1)+2) + 1}$.  
This clearly has limit $0$, so $\frac{P_{s^2} - P_s}{P_{s^2}} = 1 - \frac{P_s}{P_{s^2}}$ has limit $1$.  
\end{proof}
   
Let $P_{=m}$ be the number of identities of length exactly $m$, and let $P_{=m}(A)$ be the number of identities in $A$ of length equal to $m$.  Calculating the limit of $\frac{P_{=s}(A)}{P_{=s}}$ is often easier than calculating the limit of $\frac{P_s(A)}{P_s}$.  The lemma below gives us permission to do that.     
                                                                                                                                                                                                                                                                                                                                                                                                                                                                                                                                                                                                                                                                                                                                                                                                                                                                                                                                                                                                                                                                                                                                                                                                                                                                                                                                                                                                                                                                                                                                                                                                                                                                                                                                                                                                                                                                                                                                                                                                                                                                                                                                                                                                                                                                                                            
\begin{lem}
\label{lem4.29}

For any set $A$ of identities of arbitrary length, if $\frac{P_{=s}(A)}{P_{=s}}$ has limit $0$, then so does $\frac{P_s(A)}{P_s}$.   

\end{lem}  
  
\begin{proof}

We show that for $\epsilon > 0$, there is some $m$ such that for $s \geq m$, $\frac{P_s(A)}{P_s} < \epsilon$.
Take $m_1$ such that for all $s\geq m_1$, we have $\frac{P_{=s}(A)}{P_{=s}} < \frac{\epsilon}{2}$, and take $m_2$ such that for all $s$ such that $\sqrt{s}\geq m_2$, we have $\frac{P_{\sqrt{s}}}{P_s} < \frac{\epsilon}{2}$.   
Let $s\geq m_1,m_2$.  Then  
    
\[P_s(A) - P_{\sqrt{s}}(A) = 
\sum_{\sqrt{s} < m\leq s}P_{=m}(A) < 
\frac{\epsilon}{2} \sum_{\sqrt{s} < m\leq s}P_{=m} = \frac{\epsilon}{2} (P_s - P_{\sqrt{s}}).\] 
This gives us 
\[\frac{P_s(A)}{P_s} = \frac{P_{\sqrt{s}}}{P_s} + \frac{P_s(A) - P_{\sqrt{s}}(A)}{P_s} < \frac{\epsilon}{2} + \frac{\epsilon}{2} \cdot \frac{P_s - P_{\sqrt{s}}}{P_s} < \frac{\epsilon}{2} + \frac{\epsilon}{2} =
\epsilon.\ \]
\end{proof} 

The next lemma will complete the first step of the proof of Theorem~\ref{bij-str-thm}.  We write $t$ and $t'$ for both strings of function symbols and elements of $G(V)$.     

\begin{lem}\label{stat}

For every $k\in \mathbb{N}$, we  have 
$$\lim\limits_{s\to \infty} \frac{P_s(|\Pi_1(t) - \Pi_1(t')|< k)}{P_s} = 0\ .$$

\end{lem}

\begin{proof}

By Lemma \ref{lem4.29}, it suffices to prove that 
$$\lim\limits_{s\to \infty} \frac{P_{=s}(|\Pi_1(t) - \Pi_1(t')|< k)}{P_{=s}} = 0.$$
Furthermore, since $k$ is fixed, it is enough to prove that for every $k\in \Z$, $$\lim\limits_{s\to \infty} \frac{P_{=s}( \Pi_1(t) - \Pi_1(t')= k)}{P_{=s}} = 0.$$ 

Fix $s$.  The identities of length $s$ form a finite probability space, and the random variables $\Pi_1(t)$ and $\Pi_1(t')$ are not independent.  By Lemma \ref{number-of-id}, we may consider $\Pi_1(t) - \Pi_1(t')$ conditioned on $t$ having length $\ell$.  Then $\len(t') = s - \ell$.  For each $\ell \le s$, the number of identities with $\len(t) = \ell$ and $\len(t') = s - \ell$ is equal to the number of strings of length $s$, so the probability that $\len(t) = \ell$ is $\frac{1}{s+1}$.  The probability that $\Pi_1(t) - \Pi_1(t') = k$ is the sum over $\ell \le s$ of the probability that $\len(t) = \ell$ times the conditional probability that $\Pi_1(t) - \Pi_1(t') = k$ given $\len(t) = \ell$.  
We have
\[\frac{P_{=s}( \Pi_1(t) - \Pi_1(t') = k)}{P_{=s}} \]
\[ =\frac{1}{s+1} \sum\limits_{\ell = 0}^{s} \frac{P_{=s}( \Pi_1(t) - \Pi_1(t')= k\ \&\  \len(t) = \ell\ \& \ \len(t') = s-\ell)}{P_{=s}(\len(t) = \ell\ \&\ \len(t')= s-\ell)}.\] 

We write $X_\ell$ for $\Pi_1(t)$ conditioned on $t$ having length $\ell$. Then, as a random variable, $X_\ell$ is a sum of $\ell$ i.i.d.\ random variables $Y_{\ell_k}$ whose value is equal to the projection of the $k^{th}$ symbol.    
All function symbols are equally likely.  Thus, with probability $\frac{1}{n}$, $Y$ will be $\Pi_1(f_i)$ for $1 \le i \le n$.  As $s\to \infty$, we have $\ell\to \infty$.  By the Central Limit Theorem, we have that $X_\ell/\ell$ converges to a normal distribution. This means that, in particular, for every $\epsilon$, there is some $\ell_\epsilon$ such that for every $\ell > \ell_\epsilon$, the probability that $X_\ell = i$ is less than $\epsilon$ for all $i$; i.e., 
$$ \frac{P_{=s}( \Pi_1(t) = i\ \&\  \len(t) = \ell)}{P_{=s}(\len(t) = \ell)} < \epsilon.$$ 
Without loss of generality, we will assume that $\len(t) \ge \len(t')$, so $\ell \ge s/2$.  Thus, $\ell > \ell_\epsilon$ whenever $s > 2\ell_\epsilon$. 

Now, we have that
\begin{align*}
&\frac{P_{=s}( \Pi_1(t) -\Pi_1(t') = k \ \&\  \len(t) = \ell \ \&\  \len(t') = s-\ell)}{P_{=s}(\len(t) = \ell \ \&\  \len(t') = s-\ell)} \\
= &\sum\limits_{i} \frac{P_{=s}( \Pi_1(t) = i \ \&\  \len(t) = \ell )}{P_{=s}(\len(t) = \ell )}\cdot \frac{P_{=s}( \Pi_1(t') = i-k \ \&\  \len(t') = s-\ell)}{P_{=s}(\len(t') = s-\ell)}\\
< &\sum\limits_{i} \epsilon\cdot \frac{P_{=s}( \Pi_1(t') = i-k \ \&\  \len(t') = s-\ell)}{P_{=s}(\len(t') = s-\ell)}\\
 < &\ \ \epsilon.
\end{align*}
Combining these, we get
\begin{align*}
&\lim\limits_{s\to \infty}\frac{P_{=s}(\Pi_1(t) - \Pi_1(t')= k)}{P_{=s}} \\
= &\lim\limits_{s\to \infty}\frac{1}{s+1} \sum\limits_{\ell = 0}^{s} \frac{P_{=s}( \Pi_1(t) -\Pi_1(t') = k \ \&\  \len(t) = \ell \ \&\  \len(t') = s-\ell)}{P_{=s}(\len(t) = \ell \ \&\  \len(t') = s-\ell)} \\
= &\ \ 0.
\end{align*}

\end{proof} 

We proceed to the second step of the proof. Recall that $e_0 = \max_i |\Pi_1(f_i)|$.  

\begin{lem}\label{lift} 

Fix $R$ of the form $t(a) = t'(a)$, and fix $k$ such that \\
$|\Pi_1(t) - \Pi_1(t'))| > e_0k$. For any $u$, $v$ at a distance $\leq k$ in the Gaifman graph of $F$, we have $F\models u(a) = v(a)$ if and only if $\langle a \mid R \rangle\models u(a) = v(a)$, where $\langle a \mid R \rangle$ is the structure given by the presentation $a|R$.

\end{lem}

\begin{proof}

Based on the discussion before Lemma \ref{number-of-id}, we can see that if $u(a)$ and $v(a)$ are adjacent in the Gaifman graph, then $|\Pi_1(u) - \Pi_1(v)| = |\Pi_1(f_i)| \le e_0$ for some $f_i$. Thus, if $d(u,v)\leq k$ in the Gaifman graph, then $|\Pi_1(u) - \Pi_1(v)| \le e_0k$. We will also write $t^* = t^{-1} \circ t'$, where $t^{-1}$ is the term that is the inverse of $t$ in the theory of the commutative generalized bijective variety. Note that $t^{-1}$ may be longer than $t$, but this does not affect the argument below. 

It is easy to see that if $u(a) = v(a)$ holds in $F$, then it holds in the structure $\langle a|R\rangle$, where $R$ is $t(a) = t'(a)$, which is equivalent to $t^*(a) = a$.  Suppose
$\langle a \mid R\rangle\models u(a) = v(a)$.  By Lemma \ref{coset}, this implies that $u,v$ are in the same left coset of $\langle t^* \rangle$; i.e., $u^{-1}v \in \langle t^* \rangle$.  Taking the projection $\Pi_1$, we see that $\Pi_1(u^{-1}v) \in \Pi_1(\langle t^* \rangle)$.  For some integer $k$, we have $u^{-1}v = (t^*)^k \in \langle t^* \rangle$, and $\Pi_1((t^*)^k) = k\cdot\Pi_1(t^*)$.  However, by assumption, $|\Pi_1(t^*)| = |\Pi_1(t^{-1}t')| > e_0|k|$, and we have $|\Pi_1(u^{-1}v)| \le e_0|k|$.  Therefore, we must have $k = 0$.  It follows that $\Pi_1(u^{-1}v) = 0\cdot \Pi_1(t^*) = 0$.  Moreover, $u^{-1}v = (t^*)^0$.    
It follows that $u = v$ in $G(V)$, and $F\models u(a) = v(a)$.
\end{proof}

We are ready to complete the proof of the theorem.  We just need to show that the sentences of the form $u(a) = v(a)$ or $u(a)\not= v(a)$ true in $F$ have limiting density $1$.  By Lemma \ref{stat}, for any  integer $k$, the set of identities $t(a) = t'(a)$ such that $|\Pi_1(t) - \Pi_1(t')| >\epsilon_0|k|$ has density $1$.
For a fixed sentence $u(a) = v(a)$, take $k$ such that $u,v$ both have length at most $\frac{k}{2}$, so that $                                                                                                                                                                                                                                                                                                                                                                                                                                                                                                                                                                                                                                                                                                                                                                                                                                                                                                                                                                                                                                                                                                                                                                                                                                                                                                                                                     u(a),v(a)$ are at distance at most $k$.  Then by Lemma \ref{lift}, the sentence $u(a) = v(a)$ holds in $F$ iff it holds in the structures given by identities $t(a) = t'(a)$ such that $|\Pi_1(t) = \Pi_1(t')| > \epsilon_0 k$, where this set has density $1$.  
\end{proof}

This theorem can be generalized to presentations with multiple generators.  

\begin{prop}\label{bij-str-par}

Let $V$ be a commutative generalized bijective variety in the language $\{f_1, \cdots, f_n\}$ and suppose that the free structure on $a$ is infinite. Then for the structures in $V$ with an $m$-tuple $\overline{a}$ of generators and a single identity, the sentences true in the free structure on $\overline{a}$ have limiting density $1$.  

\end{prop}

To do so, we need the following lemma.   

\begin{lem}

Let $V$ be a commutative generalized bijective variety, with theory $T$.  Let $F_m$ be the free structure on $m$ generators.  Suppose that $F_1$ is infinite.  Then for all $m\geq 1$, $F_m$ and $F_1$ satisfy the same theory.    

\end{lem}

\begin{proof}

All elements of $F_1$ have the same local type.  
Now, $F_1$ has a saturated elementary extension $F^*$ whose Gaifman graph has infinitely many connected components.  Let $A$ be the substructure of $F^*$ extending $F_1$ and generated by an $m$-tuple $a_1,\ldots,a_m$ from different connected components.  Clearly, $F_1$ and $F^*$ satisfy the same special local sentences.  Since the sentences are existential, any special local sentence true in $F_1$ is true in $A$, and any special local sentence true in $A$ is true in $F^*$.  

We may suppose that $F_m$ has generators $a_1,\ldots,a_n$.  The connected component of $a_i$ in $F_m$ and in $A$ is generated by $a_i$---the elements are named by terms $t(a_i)$.  The special $r$-local formula $\alpha(x)\in C_{r,1}$ true of the elements of $F_1$ is true of each $a_i$ in $F_m$ and in $A$.  We have an isomorphism from $F_m$ onto $A$ that takes $a_i$ to $a_i$ and takes $B_r(a_i)$ in $F_m$ to $B_r(a_i)$ in $A$.  Then $F_1$ and $F_m$ have the same theory. 
\end{proof}               

\begin{proof} [Proof of Proposition \ref{bij-str-par}]

For presentations with $m$ generators and a single identity, we consider separately the set $M$ of presentations in which the identity involves a single generator and the complementary set $\neg{M}$ in which the identity involves two distinct generators.  
For a presentation $\bar{a}|t_1(a_i) = t_2(a_i)$ in $M$, the resulting structure is the disjoint union of the structure $\langle a_i | t_1(a_i) = t_2(a_i)\rangle$ (with generator $a_i$) and $(m -1)$ copies of $F_1$ (one for each of the other $a_j$'s).  The identities in $\neg{M}$ have the form $t_1(a_i) = t_2(a_j)$ for $i\not= j$.  In the structure $\langle \overline{a} | t_1(a_i) = t_2(a_j)\rangle$, the connected component of $a_i$ and the connected component of $a_j$ are collapsed via the relation $t(a_i) = t_1^{-1}t_2t(a_j)$. Thus, the structure is a disjoint union of $(m-1)$ copies of the free structure on one generator.    

For fixed $s$, we have a finite probability space.  For a sentence $\varphi$, the probability that $\varphi$ is true is $\frac{P_s(\varphi)}{P_s} = \frac{P_s(M\ \&\ \varphi)}{P_s} + \frac{P_s(\neg{M}\ \&\ \varphi)}{P_s}$.  
For presentations with a single generator, we write $P'_s$ and $P_s'(\varphi)$.  By Theorem \ref{bij-str-thm}, 
\[\frac{P'_s(\varphi)}{P'_s}\rightarrow 
\left\{\begin{array}{cc}
1 & \mbox{if $F_1\models\varphi$,}\\
0 & \mbox{otherwise}.
\end{array}\right.\]

Now, $\frac{P_s(M\ \&\ \varphi)}{P_s}$ is the probability of $(M\ \&\ \varphi)$.  This is equal to the probability of $M$ times the conditional probability of $\varphi$ given $M$.  The probability of $M$ is $\frac{1}{n}$.  The conditional probability of $\varphi$ given $M$ is the same as the probability of $\varphi$ for presentations with a single generator; namely, $\frac{P'_s(\varphi)}{P'_s}$.  Thus, $\frac{P_s(M\ \&\ \varphi)}{P_s} = (\frac{1}{n})(\frac{P'_s(\varphi)}{P'_s})$.     
As $s\rightarrow\infty$, this approaches $\frac{1}{n}$ if $\varphi$ is true in the free structures and $0$ otherwise. 

Similarly, the probability of $(\neg{M}\ \&\ \varphi)$ is the probability of $\neg{M}$ times the conditional probability of $\varphi$ given $\neg{M}$.  The probability of $\neg{M}$ is $\frac{(n-1)}{n}$.  The conditional probability of $\varphi$ given $\neg{M}$ is $1$ if $\varphi$ is true in $F_{m-1}$ and $0$ otherwise.  Thus, 
\[\frac{P_s(\neg{M}\ \&\ \varphi)}{P_s} = 
\left\{\begin{array}{cc}
\frac{(n-1)}{n} & \mbox{if $F_{m-1}\models\varphi$}\\
0 & \mbox{otherwise.}
\end{array}\right.\]
In total, $\frac{P_s(\varphi)}{P_s}$ has limit $\frac{1}{n} + \frac{n-1}{n} = 1$ if $\varphi$ is true in the free structures and $0$ otherwise.
\end{proof}

\begin{remark}\label{bij-str-rmk}

Using the multidimensional Central Limit Theorem \cite{vdVaart}, we can generalize the theorem and corollary above to any commutative generalized bijective variety $V$ where $\mathbb{Z}^k$ embeds into $G(V)$. In this case, the random structures in $V$ with a single generator and $k$ identities satisfy the zero--one conjecture, and the limiting theory agrees with the theory of the free structure.  However, without the condition that $\mathbb{Z}^k$ embeds in $G(V)$, the statement is false, as witnessed by the bijective structures with two identities considered in Section \ref{bij-two-id}.

\end{remark}

\subsubsection{Superstability}\label{superstable}

We make a brief comment on the superstability of completions of the theory of generalized bijective varieties. Recall that for an infinite cardinal $\kappa$, a (complete) theory $T$ is \emph{$\kappa$-stable} if for every set $A$ in a model of $T$, if $A$ has cardinality $\kappa$, then the set of complete types over $A$ has cardinality $\kappa$ as well. A theory is \emph{stable} if it is $\kappa$-stable for some $\kappa$, and it is \emph{superstable} if it is $\kappa$-stable for all sufficiently large $\kappa$.  If the language of $T$ is countable, then $\kappa\geq 2^{\aleph_0}$ will suffice. For more on stable theories, see Chapter 4 of \cite{Marker}.

\begin{prop} 

All completions of the theory of a generalized bijective variety are superstable.  

\end{prop}

\begin{proof}

Let $T$ be a completion of this theory, and let $X$ be a subset of some model of $T$ of cardinality $\kappa\geq 2^{\aleph_0}$.  We show that the number of $1$-types over $X$ is at most $\kappa$.   

A type in a variable $x$ over $X$ will say one of the following:
\begin{enumerate}

\item  $x = t(a)$ for some $a\in X$.

\item  For every term $t$ and every $a\in X$, $x\not= t(a)$, and $x$ satisfies a certain quantifier-free $1$-type $p(x)$.  

\end{enumerate}
From this, it follows that if $\kappa$ is the cardinality of $X$, then the number of $1$-types over $X$ is at most $\kappa + 2^{\aleph_0}$.  Thus, for $\kappa\geq 2^{\aleph_0}$, $T$ is $\kappa$-stable.   
\end{proof}

\begin{remark}

If we drop the condition that $T$ is a completion of the theory of a generalized bijective variety, then there are theories in a language with finitely many unary function symbols that are unstable. We will not give an example here, although one is easily obtained taking the theory of a structure in the variety which we will study in Section \ref{subsec:mult-unary-fns}.  
\end{remark}

\subsection{Failure of the zero--one law}

The next result gives conditions under which the zero--one law fails.  

\begin{thm}\label{constants-like}  

Let $L$ be a language with $n$ unary functions, including $f$, where $n\geq 2$.  Let $V$ be a variety such that for some term $t$ involving a symbol apart from $f$, the theory $T$ of $V$ contains the sentence $(\forall x)(\forall y) t(x) = t(y)$.  Consider presentations with an $m$-tuple $\bar{a}$ of generators and a single identity, and suppose that in the free structure, $f(t(a))\not= t(a)$.  Then there is a sentence with limiting density neither $0$ nor $1$.   

\end{thm}

\begin{remark} 

The sentence $(\forall x)(\forall y) t(x) = t(y)$ says that $t$ has a constant value.  If $t$ involved just the symbol $f$, then the free structure would satisfy the sentence $f(t(a)) = t(f(a)) = t(a)$. 

\end{remark}

\begin{proof} [Proof of Theorem \ref{constants-like}]

Let $\varphi$ be a sentence saying that $f$ fixes the constant given by $t$.  For instance, we may take $\varphi = (\forall x)f(t(x)) = t(x)$.  We show that $\varphi$ does not have limiting density 0 or 1.  We consider presentations with a tuple $\bar{a}$ of $m$ generators and a single identity.   
Let $A$ be the set of identities of the form $u(a_i) = v(a_j)$, where $u(a_i) = t(u'(a_i))$ and $v(a_j) = f(t(v'(a_j)))$.  In the resulting structures, $f$ fixes the constant, so $\varphi$ is true.  Let $B$ be the set of identities of the form $u(a_i) = v(a_j)$, where $u(a_i) = t(u'(a_i))$ and $v(a_j) = t(v'(a_j))$.  The resulting structure is free and $f$ does not fix the constant, so $\varphi$ is false.  We show that neither $A$ nor $B$ has density $0$.  It follows that neither $\varphi$ nor $\neg{\varphi}$ has density $0$.     

The number of identities of length $\ell$ is $m^2 n^\ell(\ell+1)$.  Then
\begin{align*}
P_s & =  m^2\sum_{0\leq \ell\leq s}(\ell+1)n^\ell = m^2(1 + 2n + \ldots + (s+1)n^s) \\
 & =  m^2\left(\frac{(s+2)n^{s+2} - (s+2)n^{s+1} + 1}{(n-1)^2}\right) = m^2\left(\frac{n^{s+1}(s+2)(n-1) + 1}{(n-1)^2}\right).
\end{align*}

Say that $t$ has length $r$.  Then the identities in $A$ have length at least $2r+1$, and for $\ell = 2r+1+\ell'$, the number of identities in $A$ of length $\ell$ is \\
$m^2 n^{\ell'}(\ell'+1)$.  Then 
\begin{align*}
P_s(A) & = m^2\sum_{2r+1+\ell'\leq s}(\ell'+1)n^{\ell'} = m^2(1 + 2n + \ldots + (s-2r)n^{s-2r-1}) \\
 &= m^2\left(\frac{n^{s-2r}(s-2r+1)(n-1) + 1}{(n-1)^2}\right),
\end{align*}
and
\[\frac{P_s(A)}{P_s} = \frac{1}{n^{2r+1}}\frac{(s-2r+1)(n-1) + 1}{(s+2)(n-1) + 1}\rightarrow\frac{1}{n^{2r+1}}\ .\]

The identities in $B$ have length at least $2r$.  For $\ell = 2r + \ell'$, the number of identities in $B$ of length $\ell$ is $m^2 n^{\ell'}(\ell'+1)$.  Then 
\begin{align*}
P_s(B) & = m^2\sum_{2r+\ell'\leq s} (\ell'+1)n^{\ell'} = m^2(1 + 2n + \ldots (s - 2r + 1)n^{s-2r}) \\
& =m^2\left(\frac{n^{s-2r+1}(s-2r+2)(n-1) + 1}{(n-1)^2}\right),
\end{align*}
and
\[\frac{P_s(B)}{P_s} = \frac{1}{n^{2r}}\frac{(s-2r+2)(n-1) + 1}{(s+2)(n-1) + 1}\rightarrow\frac{1}{n^{2r}}.\]     
Since $n\geq 2$, both of these limits are strictly between $0$ and $1$.
\end{proof}

\section{Naming the generators}  
\label{naming}

\subsection{A general result}

Let $V$ be a variety in a language $L$ with axioms generating a theory $T$.  We consider presentations with a fixed generating tuple $\bar{a}$, 
and $k$ identities.  Let $L'$ be the result of adding to $L$ constants for the generators.  We ask when the $L'$-sentences true in the free structure have limiting density $1$.  

\begin{prop} \label{par-thm}

Let $T_F$ be the set of $L'$-sentences true in the free structure $F$ generated by $\bar{a}$, and let $S$ be the set of $L'$-sentences with limiting density $1$.  Then the following are equivalent:

\begin{enumerate}

\item  $T_F\subseteq S$, 

\item $T_F = S$,

\item  $S$ has the following two properties: 

\begin{enumerate}

\item $S$ includes the sentences from $T_F$ of the form $t(\bar{a}) \not= t'(\bar{a})$, 

\item  for any $L'$-formula $\varphi(x)$ with just $x$ free, if $\varphi(t(\bar{a}))\in S$ for all closed terms $t(\bar{a})$, then $(\forall x)\varphi(x)\in S$.

\end{enumerate}

\end{enumerate}              

\end{prop}

\begin{proof}

We will prove $(1)\Rightarrow(2)\Rightarrow(3)\Rightarrow (1)$.
First, we assume (1) and prove (2).  We must show that $S\subseteq T_F$.  Take $\varphi\in S$.  If $\varphi\notin T_F$, then $\neg{\varphi}$ must be in $T_F$, so it is in $S$.  Then $\varphi$ has limiting density $0$, and we have a contradiction.  
Next, we assume (2) and prove (3).  We can see that $T_F$ has properties (a) and (b), so $S$ does as well.  
Finally, we assume (3) and prove (1).  The set $S$ has properties (a) and (b).  Sentences that are logically equivalent have the same limiting density as well as the same truth value in the free structure $F$.  We show by induction on $\varphi(\bar{a})$ that if $\varphi(\bar{a})\in T_F$, then $\varphi(\bar{a})\in S$.  We suppose that the negations in our formulas are brought inside, next to the atomic formulas.  

\begin{enumerate}

\item  Suppose $\varphi$ has the form $t(\bar{a}) = t'(\bar{a})$.  If $F\models\varphi$, then $T_F\vdash\varphi$, so the limiting density is $1$.  

\item  Suppose $\varphi$ has the form $t(\bar{a})\not= t'(\bar{a})$.  By (a), if $F\models\varphi$, then 
$\varphi$ has limiting density $1$.

\item  Suppose $\varphi = (\varphi_1\ \&\ \varphi_2)$.  
If $F\models\varphi$, then both conjuncts are true, so both have limiting density $1$.  Then $\varphi$ also has limiting density $1$.

\item  Suppose $\varphi = (\varphi_1\ \vee\ \varphi_2)$.  If $F\models\varphi$, then at least one disjunct is true, so it has limiting density $1$.  Then $\varphi$ also has limiting density $1$.

\item  Suppose $\varphi = (\exists x) \psi(x)$.  If $F\models\varphi$, then $F\models\psi(t(\bar{a}))$ for some $t(\bar{x})$.  Then this sentence has limiting density $1$, so $\varphi$ also has limiting density $1$. 

\item  Suppose $\varphi = (\forall x) \psi(x)$.  If $F\models\varphi$, then $F\models\psi(t(\bar{a}))$ for all closed terms $t(\bar{a})$.  Then the sentence $\psi(t(\bar{a}))$ has limiting density $1$ for all $t(\bar{a})$, and by (b), $(\forall x) \psi(x)\in S$. 
\end{enumerate}    
\end{proof}

Consider the following further property.

\bigskip
\noindent
\textbf{Property (c)}:  If $(\exists x)\psi(x)\in S$, then $\psi(t(\bar{a}))\in S$ for some $t(\bar{a})$.

\begin{lem}

If $S$ is complete (i.e., we have the zero--one law), then (b) and (c) are equivalent.

\end{lem}

\begin{proof}

First, suppose that (b) holds and that $(\exists x)\psi(x)\in S$.  If there is no $t(\bar{a})$ such that 
$\psi(t(\bar{a}))\in S$, then $\neg{\psi(t(\bar{a}))}\in S$ for all $t(\bar{a})$, and  
$(\forall x)\neg{\psi(x)}\in S$ for a contradiction.  Now, suppose (c) holds and that $\psi(t(\bar{a}))\in S$ for all $t(\bar{a})$.  If $\neg{(\forall x)\psi(x)}\in S$, then $(\exists x)\neg{\psi(x)}\in S$.  By (c), 
$\neg{\psi(t(\bar{a}))}\in S$ for some $t(\bar{a})$ for a contradiction, so $(\forall x)\psi(x)\in S$. 
\end{proof} 

\begin{lem}

Suppose $S$ satisfies (a) and (b).  Then for all formulas $\varphi(x,y)$ with free variables $x,y$, if 
$\varphi(t(\bar{a}),t'(\bar{a}))\in S$ for all terms $t(\bar{a}),t'(\bar{a})$, then\\
$(\forall x)(\forall y)\varphi(x,y)\in S$.

\end{lem}

\begin{proof}

For a fixed term $t(\bar{a})$, suppose $\varphi(t(\bar{a}),t'(\bar{a}))\in S$ for all $t'(\bar{a})$.  By (b), \\
$(\forall y)\varphi(t(\bar{a}),y)\in S$ for all $t(\bar{a})$.  So, by (b), $(\forall x)(\forall y)\varphi(\bar{a},x,y)\in S$.
\end{proof}       

If the orbit of $\bar{a}$ in $\mathcal{A}$ is defined by an $L$-formula $\psi(\bar{x})$, then for each $L'$-sentence $\varphi$, we have $\mathcal{A}\models\varphi(\bar{a})$ iff $\mathcal{A}$ satisfies the $L$-sentences $(\exists\bar{x})(\psi(\bar{x})\ \&\ \varphi(\bar{x}))$ and $(\forall\bar{x})(\psi(\bar{x})\rightarrow\varphi(\bar{x}))$.  

\begin{prop}
\label{parameters}

Let $F$ be the free structure generated by $\bar{a}$.  Suppose that the orbit of $\bar{a}$ is defined by the $L$-formula $\psi(\bar{x})$ and the $L'$-sentence $\psi(\bar{a})$ has limiting density $1$.  Suppose also that for all $L$-sentences $\varphi$, $\varphi$ is true in $F$ if and only if it has limiting density $1$.  Then the same is true for all $L'$-sentences.    
 
\end{prop}

\begin{proof}

Take an $L'$-sentence $\varphi(\bar{a})$ that is true in $F$.  In $F$, this is equivalent to the $L$-sentence $\varphi^* = (\forall\bar{x})(\psi(\bar{x})\rightarrow\varphi(\bar{x}))$.  The sentence $\varphi^*$ is true in $F$, so it has limiting density $1$.  The set of sentences with limiting density $1$ is closed under logical consequence, so since $\psi(\bar{a})$ has limiting density $1$, it follows that $\varphi(\bar{a})$ has limiting density $1$.
\end{proof}  

\subsection{Generalized bijective structures and sentences with constants}

We have seen that for the basic bijective variety and for the broader class of commutative generalized bijective varieties, when we consider presentations with a single generator and a single identity, the sentences (in the language of the variety) true in the free structure have density $1$.  We can apply Proposition~\ref{par-thm} to extend this to sentences with a constant naming the generator.         

\begin{example}

Let $V$ be a commutative generalized bijective variety in the language $L$.  Consider presentations with a single generator $a$ and a single identity, and let $L'$ be the extension of $L$ with a constant for the generator.  Suppose that the free structure $F$ generated by $a$ is infinite.  In $F$, all elements are automorphic. In particular, $a$ and $t(a)$ are automorphic via the automorphism $x\mapsto t(x)$. Preparing to apply Proposition \ref{par-thm}, we take $\psi(x)$ to be $x = x$.  Clearly, $\psi(a)$  has limiting density $1$.  By Theorem \ref{bij-str-thm}, the $L$-sentences true in the free structure have limiting density $1$. Then  Proposition \ref{par-thm} says that this holds also for the $L'$-sentences (involving $a$).   

\end{example}

For a generating tuple $\bar{a}$, the sentences $\varphi(\bar{a})$ true in the free structure on $\bar{a}$ have density $1$.  To establish this, we need to take a closer look at the formulas $C_{r,n}$ from Section \ref{Gaif-sec} in the expanded language $L'$ with constants naming the generators.  We consider the unary functions $f_i$ as binary relations and the constants $a_i$ as unary relations.  Thus, we have atomic formulas with the meanings $x = a_i$ and $f_i(x) = y$. 

\begin{lem}\label{r-types}

Let $V$ be a commutative generalized bijective variety, and consider presentations with a generating $m$-tuple $\bar{a}$ and a single identity.  Let $F$ be the free structure.  If $F$ is infinite, then for every $r\in \omega$, there is a set $S$ of presentations such that 

\begin{enumerate}

\item $S$ has limiting density $1$, and 

\item for $\alpha(\bar{x})\in C_{r,m}$, the following are equivalent:

\begin{enumerate}

\item  $\alpha(\bar{a})$ holds in $F$,

\item  $\alpha(\bar{a})$ holds in some structure given by a presentation in $S$, 

\item  $\alpha(\bar{a})$ holds in all structures given by a presentation in $S$.  

\end{enumerate}   

\end{enumerate}
If $\mathcal A$ is the structure given by the identity $t_1(a_i) = t_2(a_j)$, then for each $r$, we get an isomorphism $p$ from $B_{2r}(\bar{a})$ in $F$ to $B_{2r}(\bar{a})$ in $\mathcal{A}$, given by $p(u(a_k)) = u(a)$.

\end{lem}

\begin{proof} 

We use notation from the proof of Theorem \ref{bij-str-thm}. Recall that $\Pi_1$ is the projection onto the copy of $\mathbb{Z}$ generated by $b_1$, where $b_1$ is an element of infinite order in the abelian group $G(V)$ associated with the variety $V$.  Let $S$ be the set of presentations in which the identity $t_1(a_i) = t_2(a_j)$ satisfies $|\Pi_1(t_1) - \Pi_1(t_2)| > e_0r$, where $e_0 = \max_i |\Pi_1(f_i)|+4$. The fact that $S$ has limiting density $1$ follows from the proof of Theorem \ref{bij-str-thm}.

Since the formulas in $C_{r,m}$ uniquely describe the isomorphism type of $B_r(\bar x)$, it suffices to show that $p$ is an isomorphism. We know that $p$ is surjective since $\mathcal A$ is a quotient of $F$.  As in the proof of Theorem \ref{bij-str-thm}, it is also injective. Indeed, if $i= j$, then the projection from $F$ to $\mathcal A$ is injective on the substructure generated by $a_1, \cdots, a_{i-1}, a_{i+1},\cdots, a_m$. On the substructure generated by $a_i$, if $u(a_i) = u'(a_i)$, then we can apply Lemma \ref{coset}, and we see that  (the equivalence class of) $u^{-1}u'$ is in $\langle t_1t_2^{-1} \rangle$ as an element of $G(V)$. Since $|\Pi_1(t_1) - \Pi_1(t_2)| > e_0r$ and the length of $u^{-1}u'$ is at most $4r$, this is not possible. If $i \neq j$, then, as in Corollary \ref{bij-str-par}, the substructure generated by $a_i$ and that generated by $a_j$ are identified via $a_j = t_1t_2^{-1}(a_i)$, while the projection map is injective on the substructure generated by the further generators $a_k$. Thus, if $u(a_i) = u'(a_j)$, then we must have $uu'^{-1} = t_1t_2^{-1}$, but since $|\Pi_1(t_1) - \Pi_1(t_2)| > e_0r$ and the length of $u^{-1}u'$ is at most $4r$, this is again impossible. 

Recall that we are thinking of the language as relational and we have atomic formulas with the meanings $x = a_i$ and $f_i(x) = y$.  The formula saying $x = a_i$ holds exactly on $a_i$, in either $B_{2r}(\bar a;F)$ or $B_{2r}(\bar a; \mathcal A)$.  If the formula saying $f_i(x) = y$ holds in $B_{2r}(\bar a;F)$, then it holds of $p(x)$ and $p(y)$ in $B_{2r}(\bar a; \mathcal A)$ because $\mathcal A$ is a quotient of $F$.  Thus, we have $f_i(p(x)) = p(y)$.  Finally, suppose that in $B_{2r}(\bar{a};\mathcal{A})$, $f_i(p(x)) = p(y)$. Then $p(x) = u(a_j)$ for some $j$ and $p(y) = f_i(u(a_j))$.  However, the map $p$ is bijective, so $F\models x = u(a_j)$ and $F\models y = f_i(u(a_j))$ as well.  Then $f_i(x) = y$ holds in $B_{2r}(\bar a;F)$. This shows that $p$ is an isomorphism, completing the proof. 
\end{proof}

\begin{thm}\label{gen-bij-par}

Let $V$ be a commutative generalized bijective variety in the language $L$, and suppose that the free structure on one generator is infinite.  Consider presentations with a fixed generating $m$-tuple $\bar a$ and a single identity.  Let $F$ be the free structure on $\bar{a}$.  Let $L'$ be the result of adding constants for the elements of $\bar{a}$ to $L$.  Then an $L'$-sentence is true in $F$ iff it has limiting density $1$.   

\end{thm}

\begin{proof}

Let $\varphi'$ be an $L'$-sentence that is true in $F$, so $\varphi' = \varphi(\bar a)$ for some $L$-formula $\varphi$.  By Theorem \ref{Gaifman}, $\varphi(\bar x)$ can be expressed as a finite disjunction $\bigvee_i \varphi_i(\bar x)$ for $\varphi_i(\bar x)$ of the form $\rho_i(\bar x)\ \&\ \chi_i $, where $\rho_i(\bar x) \in C_{r,m}$ and $\chi_i$ is a conjunction of special sentences and negations of special sentences.  Recall that the special sentences have the form $(\exists v_1, \cdots, v_s) \left( \bigwedge\limits_{i} \alpha_i(v_i) \ \& \bigwedge\limits_{i<j}d^{>2r}(v_i,v_j)\right)$, where $\alpha_i(v_i)\in C_{r,1}$ and $d^{>2r}(v_i,v_j)$ is the formula saying that the distance between $v_i$ and $v_j$ in the Gaifman graph is greater than $2r$.  Since $F \models \varphi(\bar a)$, we have $F \models \varphi_i(\bar a)$ for all $i$, i.e., $F \models \rho_i(\bar a)\ \&\ \chi_i$. Since $\chi_i$ is an $L$-sentence true in $F$, it has limiting density $1$ by Theorem~\ref{bij-str-thm}.  On the other hand, $\rho_i(\bar x)\in C_{r,m}$, and by Lemma \ref{r-types}, $\rho_i(\bar a)$ also has limiting density $1$. Thus, $\varphi'$ has limiting density $1$.
\end{proof}  

\section{More examples}\label{sec:MoreExamples}

In Section \ref{illu-ex}, we gave some examples illustrating different possible behaviors of limiting density.  We considered sentences with no constants.  In Section \ref{general}, we gave conditions guaranteeing that the sentences with limiting density $1$ are those true in the free structure.  In Section \ref{naming}, we gave some results for sentences with constants naming the generators.  In the current section, we look again at some of the examples from Section \ref{illu-ex} in light of the results from Sections \ref{general} and \ref{naming}.  We we also give some further examples, illustrating more subtle points suggested by these results.     

\subsection{Examples of Proposition~\ref{par-thm}}

Let $V$ be a variety in the language $L$.  Consider presentations with a fixed tuple $\bar{a}$ of generators, and some number of identities, and let $L'$ be the result of adding to $L$ constants for the generators.  Here, for reference, is the statement of Proposition~\ref{par-thm}.

\bigskip
\noindent
\textbf{Proposition \ref{par-thm}}.  Let $T_F$ be the set of $L'$-sentences true in the free structure $F$ generated by $\bar{a}$, and let $S$ be the set of $L'$-sentences with limiting density $1$.  Then the following are equivalent:

\begin{enumerate}

\item  $T_F\subseteq S$, 

\item $T_F = S$,

\item  $S$ has the following two properties: 

\begin{enumerate}

\item $S$ includes the sentences from $T_F$ of the form $t(\bar{a}) \not= t'(\bar{a})$, 

\item  for any $L'$-formula $\varphi(x)$ with just $x$ free, if $\varphi(t(\bar{a}))\in S$ for all closed terms $t(\bar{a})$, then $(\forall x)\varphi(x)\in S$.

\end{enumerate}

\end{enumerate}              

The proposition says that conditions (a) and (b) are necessary and sufficient for the $L'$-sentences true in $F$ to have density $1$.  We revisit some examples and see what the result says about them.

\subsubsection{Generalized bijective structures} 

In Theorem~\ref{gen-bij-par}, we saw that for the variety of generalized bijective structures and presentations with a single generator $a$ and a single identity, any sentence, possibly involving the constant $a$, has limiting density $1$ iff it is true in the free structure on $a$.  Hence, we must have both properties (a) and (b) from Proposition~\ref{par-thm}.

\subsubsection{Abelian groups}

In Section~\ref{abelian groups}, we saw that for abelian groups and presentations with a single generator and a single relator, the zero--one law fails.  

\begin{prop}

For abelian groups and presentations with a single generator and a single relator, Property (a) holds and Property (b) fails, witnessed by the formulas $\varphi(a,x)$ saying that $p^{n+1}x \not= p^na$, where $p$ is an odd prime. 

\end{prop}

\begin{proof}

The free structure is $\Z$.  Take a sentence of the form $ma\not= 0$.  This is true in $\Z$, and the sentence is in $S$ since all relators longer than $|m|$ make it true.  Thus, Property (a) holds.  Now fix $n$ and an odd prime $p$.  For all closed terms $t(a)$, the sentence $p^{n+1}t(a)\not= p^n(a)$ is in $S$.   
We have $p^{n+1}t(a)\not= p^na$ for all terms $t(a) = ma$.  By Property (a), the sentences $p^{n+1}t(a)\not= p^na$ are all in $S$. If we had Property (b), then the sentence $(\forall x) p^{n+1}x\not= p^na$ would be in $S$. However, recall from Section \ref{abelian groups} that the sentence $\beta(p,n,1)$ says that there is an element divisible by $p^n$ and not by $p^{n+1}$.  This is true in $\Z$ but not in $S$---by Lemmas \ref{beta-condition} and \ref{beta-prob}, the limiting density is $\frac{1}{p^{n+1}}$.  Since $p^na = p^na$ is logically valid, $p^na$ is divisible by $p^n$ in all models.  Thus, if the sentence $(\forall x) p^{n+1}x\not= p^na$ is in $S$, then $\beta(p,n,1)$ is in $S$---in the models satisfying $(\forall x) p^{n+1}x\not= p^na$, $p^na$ is not divisible by $p^{n+1}$.  This is a contradiction.
\end{proof}  

\subsubsection{Structures with a single unary function, one generator, and one identity}

The next example is from Section \ref{basic-single-unary}.  The variety of unary functions has a single unary function symbol $f$ and no axioms.      

\begin{prop}

For the variety of unary functions and presentations of the form $a|f^m(a) = f^n(a)$, Property (a) holds and Property (b) fails. 

\end{prop}

\begin{proof}   

To show that Property (a) holds, consider a sentence of the form $f^i(a) \neq f^j(a)$.  Note that the set of presentations with identity $f^n(a) = f^m(a)$ for $n,m > i+j+1$ has limiting density $1$.  Moreover, for any such presentation, we recall from Section \ref{basic-single-unary} that the resulting structure is a finite chain leading to a finite cycle where the chain is longer than both $i+1$ and $j+1$.  Then we have $f^i(a) \neq f^j(a)$ in the structure.  Thus, $f^i(a) \neq f^j(a)$ is in $S$. 

To show that Property (b) fails, let $\varphi(x) = (\forall y)(x\not= y\rightarrow f(x)\not= f(y))$. We will show that this witnesses the failure of (b). For any fixed $x = f^i(a)$, note that the set of presentations with identity $f^n(a) = f^m(a)$ for $n,m > i+1$ has limiting density $1$. In any such presentation, we have $f^{j+1}(a) \neq f(y)$ unless $y = f^j(a)$.  Thus, the sentence saying that the formula $$\varphi(x) = (\forall y)(x\not= y\rightarrow f(x)\not= f(y))$$ holds for $x = f^i(a)$ is in $S$ for any closed term $f^i(a)$. On the other hand, the sentence $(\forall x)(\forall y)(x\not= y\rightarrow f(x)\not= f(y))$ saying that $f$ is injective has limiting density $0$, as shown in Section \ref{basic-single-unary}.  Thus Property (b) fails in this variety. 
\end{proof} 

\subsubsection{A new example}

In the next example, we modify the variety of bijective structures to obtain an example in which Property (b) holds but Property (a) fails.  

\begin{example}

Let $L_c$ be the language that consists of unary function symbols $S$, $S^{-1}$ and a constant $c$, and let $V$ be the variety with axioms saying that $S$ and $S^{-1}$ are inverse functions and $S^3(c) = c$.  Consider presentations with a single generator $a$ and one identity.  For the resulting structure $A$, let $A_a$ be the cycle generated by $a$ and let $A_c$ be the cycle generated by $c$.  We describe the structures obtained from all possible identities, and we give some limiting densities.      

\begin{enumerate}

\item  Let $S_1$ be the set of identities of the form $S^n(a) = S^m(a)$.  This has density $\frac{1}{4}$.  If $k = |m - n|$, then $A_a$ is a $k$-cycle if $k > 0$ and a $\Z$-chain if $k = 0$. $A_c$ is always a $3$-cycle in this case.

\item  Let $S_2$ be the set of identities of the form $S^n(c) = S^m(c)$.  This has density $\frac{1}{4}$.  Then $A_a$ is a $\Z$-chain (always the same), and $A_c$ is a $3$-cycle or a $1$-cycle.     

\item  Let $S_3$ be the set of identities of the form $S^n(a) = S^m(c)$ or $S^n(c) = S^m(a)$.  This set has density $\frac{1}{2}$.  In the resulting structure, $A_a = A_c$ is a $3$-cycle.

\end{enumerate}

To see that Property (a) fails, consider the sentence $c \neq S(c)$. This is true in the free structure but fails exactly in the subset of $S_2$ where $A_c$ is a 1-cycle, which has limiting density $\frac16$.

To show that Property (b) holds, assume for some $\varphi(x)$, $\varphi(t)$ has limiting density 1 for all closed terms $t = t(c,a)$. We will show that for $i = 1,2,3$, the set of identities in $S_i$ for which the resulting structure satisfies $(\forall x)\varphi(x)$ has the same density as $S_i$.  For any finite set $\sigma$ of closed terms $t = t(c)$ or $t = t(a)$, the sentence $\psi(c,a) = \bigwedge_{t\in\sigma} \varphi(t)$ has density $1$.  This makes the case $S_3$ easy.  For the structures given by identities in $S_3$, all $x$ are named by terms $c$, $S(c)$, or $S^2(c)$.   

For the remaining cases, we use Gaifman's Locality Theorem.  Consider a formula $\varphi'(u,v,x)$ in the language of bijective structures such that $\varphi(x) = \varphi'(c,a,x)$.  By Gaifman, $\varphi'(u,v,x)$ is equivalent in bijective structures to a formula $\bigvee_i(\alpha_i(u,v,x)\ \&\ \beta_i)$, where for some $r$, $\beta_i$ is a conjunction of local sentences and negations, each $r'$-local for $r'\leq r$, and $\alpha_i(u,v,x)$ is an $r$-local formula that describes the union of the $r$-balls around $u,v,x$.  

For identities in $S_2$, $A_c$ may have one element or three, and $A_a$ is fixed.  Let $\sigma$ be a finite set with closed terms naming the elements of $A_c$ and the elements of $A_a$ that are not far from $a$, with $d(x,a)\leq 2r$, plus one more element $x = t^*(a)$ where $d(x,a) > 2r$.  The sentence $\psi(a,c)$ saying that $\varphi(t)$ holds for all of these terms has density $1$.  For the other $x\in A_a$, the ones far from $a$, the balls $B_r(x)$ are isomorphic.  If $t^*(a)$ satisfies $\alpha_i(c,a,x)$, then all elements do, so $\varphi(x)$ holds.   Then $(\forall x)\varphi(x)$ has density $1$.       

Finally, for an identity in $S_1$, $A_c$ is a fixed $3$-cycle, while $A_a$ varies with the identity.  Consider the disjuncts $\alpha_i(c,a,x)\ \&\ \beta_i$ that might be satisfied by some $x\in A_a$.  The same identities also yield plain bijective structures $A_a$.  Let $\alpha'_i(a,x)$ be the part of $\alpha_i(c,a,x)$ describing the $r$-balls around $a$ and $x$.  For $x = t(a)$, $x$ satisfies $\alpha_i(c,a,x)\ \&\ \beta_i$, iff $\beta_i$ holds in $A_a\cup Z_3$ and $\alpha'_i(a,x)$ holds in $A_a$. For each $\beta_i$, there is a sentence $\beta'_i$ such that for the structures $A$ given by an identity in $S_1$, $A\models \beta_i$ iff $A_a\models \beta_i'$.  We may take $\beta'_i$ to be a finite disjunction of conjunctions of sentences that, in the setting of bijective structures, are $r'$-local sentences or negations.  We can see this by using the Feferman-Vaught Theorem or, less formally, just by thinking about what $\beta_i$ says.  Let $\varphi'(a,x) = \bigvee_i\alpha'_i(a,x)\ \&\ \beta'_i$.  For $x = t(a)$, the formula $\varphi'(a,x)$ has density $1$ in the bijective structure $A_a$.  Then by our earlier result, $(\forall x)\varphi'(x)$ has density $1$.  For each bijective structure generated by $a$ in which $(\forall x)\varphi'(x)$ is true, we consider the structure to be $A_a$, and in the variety we are currently considering, $A = A_a\cup Z_3$ satisfies $(\forall x)\varphi(x)$, so this has density $1$.  

\end{example}              

This example also shows that the zero--one law may fail if we allow constants in the language, giving an obstacle for generalizing Theorem \ref{bij-str-thm} (on the zero--one law for generalized bijective structures) to varieties in a language with constants.  Note that in Section~\ref{naming}, we did add constants naming a tuple of generators.  However, these constants were not part of the language of the variety---they did not appear in the axioms. 

\subsection{Structures with a single unary function and more generators and identities}

For the language with a single unary function symbol $f$ and the variety with no axioms, we saw in Section \ref{basic-single-unary} that for presentations with a single generator and a single identity, the zero--one law holds, but the limiting theory is not that of the free structure.  In particular, the sentence $\varphi$ saying that $f$ is not injective has density $1$, but it is false in the free structure. We now consider presentations with multiple generators and identities.

\begin{prop}

Let $L$ be the language with a single unary function symbol $f$, and let $V$ be the variety with no axioms.  For presentations with $m$ generators and $k$ identities, the sentence $\varphi = (\exists x,y)(f(x) = f(y)\ \&\ x\neq y)$ has density $1$.

\end{prop}  

Let the generators be $a_1,\ldots,a_m$.  The identities have the form $f^p(a_i) = f^q(a_j)$.  As before, the sentence $\varphi$ is true if the chosen identities all satisfy that $p,q$ are both non-zero and $p\neq q$. Indeed, in this case, without loss of generality, we may take $n,i$ such that $f^n(a_i)$ appears as one side of some identity and there is no $m<n$ such that $f^m(a_i)$ appears as one side of some identity. Suppose $f^n(a_i) = f^q(a_j)$ is one of the identities. Then $x = f^{n-1}(a_i)$ and $y = f^{q-1}(a_j)$ witness $\varphi$.  

We can show that $\varphi$ has density $1$.  The number of identities of length $r$ is $m^2(r+1)$.  The number of length at most $s$ is $m^2(1+2+\cdots + (s+1)) = m^2\frac{(s+2)(s+1)}{2}$.  The number of unordered sets of $k$ identities of length at most $s$ is $P_s = \left(\begin{array}{c}
m^2\frac{(s+2)(s+1)}{2}\\k\end{array}\right)$.  We count the identities of length $r$ such that $p,q\neq 0$ and $p\neq q$.  If $r$ is even, then there are at most $3m^2$ identities of length $r$ for which the condition fails; namely, $f^r(a_i) = a_j$, $a_i = f^r(a_j)$, and $f^{r/2}(a_i) = f^{r/2}(a_j)$.  (If $r$ is odd, then the number is at most $2m^2$.)  

Thus, there are at least $m^2(r-2)$ identities of length $r$ satisfying the condition, and there are at least $m^2\frac{(s-1)(s-2)}2$ identities of length at most $s$ satisfying the condition. Let $A$ be the set of presentations with all identities satisfying the condition.  Then $P_s(A) \ge \left(\begin{array}{c}m^2\frac{(s-1)(s-2)}{2} \\k\end{array}\right)$.        

It is now a calculus exercise to show that $\frac{P_s(A)}{P_S} \to 1$, and the proposition follows.

f

\begin{remark}

We saw that when $m = k = 1$, the zero--one law holds. However, it does not hold in the case where $m = 1$ and $k = 2$.  Suppose the two identities are $f^p(a) = f^q(a)$, $f^{p'}(a) = f^{q'}(a)$, and consider the sentence $\psi = (\exists x) f(x)~=~x$. This case is similar to the case of bijective structures with two identities in Section \ref{bij-two-id}.  The sentence $\psi$ is true if and only if $GCD(p-q,p'-q') = 1$. An argument like that in Section \ref{bij-two-id} shows that $\psi$ has density strictly between $0$ and $1$.  We omit the proof here. 

\end{remark}

\subsection{Structures with multiple unary functions}\label{subsec:mult-unary-fns}

We turn our attention to a more complicated case.  Take the language with $n$ function symbols $f_1,\ldots,f_n$ and the variety with no axioms, and consider presentations with $m$ generators and $k$ identities.  We begin with the case where $k = 1$.     

\begin{prop}\label{prop:FinFreeStruct}

Let $\varphi$ be the sentence 
\[(\exists x)(\exists y)\left(x\not= y\ \&\ \bigvee_{1\leq i,j\leq n} f_i(x) = f_j(y)\right)\ .\]  
This sentence is false in the free structure, but it has limiting density $1$ among structures given by presentations with generators $a_1, \ldots, a_m$ and a single identity of the form $t(a_i) = t'(a_j)$.

\end{prop}

\begin{proof}

For $m$ generators $a_1,\ldots,a_m$, the free structure $F$ is the join of disjoint substructures generated by the separate $a_j$.  In $F$, each element is uniquely expressed as $t(a_i)$, where the term $t$ is built up out of the functions $f_j$.  The terms are all distinct, and the sentence $\varphi$ is false.  For an identity $t(a_i) = t'(a_j)$, the length is the sum of the lengths of $t,t'$.  The number of identities of length $\ell$ is $m^2 n^\ell (\ell+1)$, so the number of identities of length at most $s$, or $P_s$, is $m^2\sum_{0\leq \ell\leq s}n^\ell(\ell+1)$, which is equal to $m^2\frac{(n-1)(s+2)n^{s+1} - (n^{s+2} - 1)}{(n-1)^2}$.

Let $A$ be the set of identities $t(a_i) = t'(a_j)$ such that $t$ has length $0$.  We show that $A$ has limiting density $0$.  The number of identities in $A$ of length $\ell$ is $m^2 n^\ell$, so the number of length at most $s$ is $m^2(\sum_{0\leq \ell\leq s}n^\ell)$, or $m^2\frac{n^{s+1} - 1}{n - 1}$.  This is $P_s(A)$.  It is not difficult to verify that $\lim_{s\rightarrow\infty}\frac{P_s(A)}{P_s} = 0$.  Similarly, let $B$ be the set of identities $t(a_i) = t'(a_j)$ such that $t'$ has length $0$.  Then $B$ also has limiting density $0$.  Therefore, the limiting density of $A\cup B$ is $0$.  Let $C$ be the set of identities not in $A\cup B$.  This will have limiting density $1$.  The identities in $C$ have the form $t(a_i) = t'(a_j)$ where $t,t'$ both have length at least $1$.  Say that $t(a_i) = f_{i'}(t^*(a_i))$ and $t'(a_j) = f_{j'}(t'^*(a_j))$ for terms $t^*$ and $t'^*$.  In the model given by the identity $t(a_i) = t'(a_j)$, we have $t^*(a_i)\not= t'^*(a_j)$.  The elements $x = t^*(a_i)$ and $y = t'^*(a_j)$ witness that the sentence $\varphi$ is true. 
\end{proof}

Now, we consider presentations with more than one identity.  We let $\varphi$ be as in Proposition \ref{prop:FinFreeStruct}.      

\begin{prop}

For the language with $n$ unary function symbol $f_1, \dots, f_n$, let $V$ be the variety with no axioms.  For presentations with a fixed $m$-tuple of generators and $k$ identities, where $k\geq 2$, the sentence $\varphi$ has limiting density $1$.

\end{prop}

\begin{proof}

Let $I_s$ be the number of identities of length at most $s$.  Then the number of presentations in which all identities have length at most $s$ is 
$P_s = \left(\begin{array}{c}
I_s\\k\end{array}\right)$.
Consider the identities $t(a_i) = t'(a_j)$ in which neither side has length $0$.  The number of these identities of length $\ell$, where $\ell\geq 2$, is $m^2 n^\ell(\ell-1)$, so the number of length at most $s$ is 
$m^2\sum_{2\leq \ell\leq s}n^\ell(\ell-1) = m^2\frac{(n-1)sn^{s-1} - (n^s - 1))}{(n-1)^2}$.  For convenience, we call this $J_s$.  Let $C$ be the set of presentations with $k$ identities in which neither side has length $0$.  Then $P_s(C) =  
\left(\begin{array}{c}
J_s\\k\end{array}\right)$.    

We show by induction on $k$ that $\lim_{s\rightarrow\infty}\frac{P_s(C)}{P_s} = 1$.  We write $P_s^k$ and $P_s^k(C)$ to indicate the value of $k$ under consideration.  For $k = 2$, $\frac{P_s^2(C)}{P_s^2} = \frac{J_s(J_s - 1)}{I_s(I_s - 1)}$.  We know that $\frac{J_s}{I_s} \rightarrow 1$.  For the expression $\frac{P_s^2(C)}{P_s^2}$, we divide top and bottom both by $I_s$ and get a new numerator $\frac{J_s}{I_s} - \frac{1}{I_s}$ that goes to $1$ and a new denominator $\frac{I_s}{I_s} - \frac{1}{I_s}$ that also goes to $1$.  Now, supposing that the statement holds for $k$, we show that it holds for $k+1$.  We have $\frac{P_s^{k+1}(C)}{P_s^k} = \left(\frac{P_s^k(C)}{P_s^k}\right)\left(\frac{J_s - k}{I_s - k}\right)$.  By the Induction Hypothesis, the first factor goes to $1$.  For the second factor, we again divide top and bottom by $I_s$.  The new numerator is $\frac{J_s}{I_s} - \frac{k}{I_s}$, which has limit $1$.  The new denominator is $\frac{I_s}{I_s} - \frac{k}{I_s}$, which also has limit $1$.

We claim that the sentence $\varphi$ is true in all structures obtained from presentations in $C$.  Take any presentation in $C$ and consider the resulting model $\mathcal{A}$.  No $a_i$ is in the range of any function in any model of this sort.  The given identities all take us from a non-trivial term in some $a_i$ to a non-trivial term in some $a_j$ and do not force us to assign values $a_i$, so we can fill out the rest of the function values without ever using these values $a_i$. Thus, all nontrivial identities true in $\mathcal{A}$ are in all structures with presentations in $C$.  Take an identity of shortest length, say $t(a_i) = t'(a_j)$, and proceed as for a single identity. Say that $t(a_i) = f_{i'}(t^*(a_i))$ and $t'(a_j) = f_{j'}(t'^*(a_j))$ for terms $t^*$ and $t'^*$. By the minimality of the length of $t(a_i) = t'(a_j)$, we have $t^*(a_i)\not= t'^*(a_j)$.  This witnesses the truth of $\varphi$.
\end{proof}

\bibliographystyle{plain}
\bibliography{zero_one}

\end{document}